\documentclass[12pt]{article}
\usepackage{geometry} 
\usepackage{amsmath}
\usepackage{amssymb}
\usepackage{mathabx}
\usepackage{amscd}
\usepackage{latexsym}
\usepackage{amsthm}
\usepackage{mathrsfs}
\usepackage{color}
\usepackage{amsfonts}
\usepackage{array}
\usepackage{setspace}
\usepackage{algorithm}
\usepackage{algpseudocode}
\usepackage{mathtools}
\usepackage{tikz}
\usepackage{tikz-cd}
\usepackage{hyperref}
\usetikzlibrary{arrows}
\usetikzlibrary{knots}
\usepackage{ytableau}
\usepackage{stmaryrd,graphicx}
\usepackage{cancel}
\usepackage{algpseudocode}
\usepackage[shortlabels]{enumitem}
\usepackage{soul}

\newtheorem{thm}{Theorem}[section]
\newtheorem*{thm*}{Theorem}
\newtheorem{cor}[thm]{Corollary}
\newtheorem*{cor*}{Corollary}
\newtheorem{lem}[thm]{Lemma}
\newtheorem*{lem*}{Lemma}
\newtheorem{prop}[thm]{Proposition}
\newtheorem*{prop*}{Proposition}

\newtheorem*{thma}{Theorem A}
\newtheorem*{thmb}{Theorem B}
\newtheorem*{thmc}{Theorem D}
\newtheorem*{corc}{Corollary C}
\theoremstyle{definition}
\newtheorem{defn}[thm]{Definition}
\newtheorem*{defn*}{Definition}

\newtheorem{conjecture}[thm]{Conjecture}
\newtheorem*{conjecture*}{Conjecture}

\newtheorem*{conja}{Conjecture E}
\newtheorem*{condition*}{Condition}
\newtheorem*{assumption*}{Assumption}

\newtheorem{algo}{Algorithm}

\theoremstyle{remark}
\newtheorem{rem}[thm]{Remark}
\newtheorem*{rem*}{Remark}
\newtheorem{example}[thm]{Example}

\newtheorem*{problem*}{Problem}

\newcommand{\Ht}{\widetilde{H}}
\newcommand{\N}{\mathbb{N}}
\newcommand{\Z}{\mathbb{Z}}
\newcommand{\Q}{\mathbb{Q}}
\newcommand{\C}{\mathbb{C}}

\newcommand{\xx}{\mathbf{x}}
\newcommand{\ax}{\mathbf{a}}

\newcommand{\bx}{\mathbf{b}}
\newcommand{\cx}{\mathbf{c}}
\newcommand{\dx}{\mathbf{d}}
\newcommand{\ix}{\mathbf{i}}
\newcommand{\osp}{\mathcal{OSP}}
\newcommand{\AAA}{\mathcal{A}}
\newcommand{\BB}{\mathcal{B}}
\newcommand{\CC}{\mathcal{C}}
\newcommand{\DD}{\mathcal{D}}
\newcommand{\AAAA}{\mathbf{A}}
\newcommand{\BBB}{\mathbf{B}}

\newcommand{\DX}{\mathfrak{D}}

\newcommand{\nls}{_{n,\lambda,s}}
\newcommand{\sym}{\mathfrak{S}}

\newcommand\numberthis{\addtocounter{equation}{1}\tag{\theequation}}

\DeclareMathOperator{\maj}{maj}

\DeclareMathOperator{\inv}{inv}

\DeclareMathOperator{\Des}{Des}
\DeclareMathOperator{\Asc}{Asc}

\DeclareMathOperator{\height}{ht}

\DeclareMathOperator{\Hilb}{Hilb}
\DeclareMathOperator{\Frob}{Frob}
\DeclareMathOperator{\rw}{rw}
\DeclareMathOperator{\des}{des}
\DeclareMathOperator{\lex}{lex}
\DeclareMathOperator{\content}{content}
\DeclareMathOperator{\sort}{sort}

\DeclareMathOperator{\sh}{sh}

\DeclareMathOperator{\ctype}{ctype}

\DeclareMathOperator{\cocharge}{cocharge}
\DeclareMathOperator{\sgn}{sgn}

\DeclareMathOperator{\im}{im}
\DeclareMathOperator{\cc}{cc}

\DeclareMathOperator{\Inj}{Inj}
\DeclareMathOperator{\rev}{rev}
\DeclareMathOperator{\col}{col}

\DeclareMathOperator{\weight}{weight}
\DeclareMathOperator{\snake}{snake}
\DeclareMathOperator{\shape}{sh}
\DeclareMathOperator{\std}{std}

\DeclareMathOperator{\Sp}{Sp}

\DeclareMathOperator{\SSYT}{SSYT}
\DeclareMathOperator{\SYT}{SYT}
\newcommand{\stirling}[2]{\genfrac{[}{]}{0pt}{}{#1}{#2}}

\title{Representation Theoretic Bases for the $\Delta$-Springer Module}
\author{Raymond Chou, Mitsuki Hanada}

\begin{document}

\maketitle

\begin{abstract}
    We give a descent monomial basis of $\Delta$-Springer modules $R\nls$, first defined by Griffin in \cite{griffin2021ordered}. Our construction simultaneously generalizes the descent basis for the Garsia-Procesi module $R_\lambda$ \cite{carlsson2024descent}\cite{hanada2025charge}, as well as the descent basis for the generalized coinvariant algebras $R_{n,k}$ studied in \cite{HaglundRhoadesShimozono2018}. This basis is deeply connected with a combinatorial object called battery-powered tableaux, introduced by Gillespie--Griffin \cite{griffin2024delta_springer}. We highlight the representation theoretic properties of this monomial basis by using it to give a direct combinatorial proof of the graded Frobenius character of $R\nls$ in terms of battery-powered tableaux, a fact which has only the geometric proof of Gillespie-Griffin. 
    We also conjecture a higher Specht basis of $R\nls$,  generalizing the higher Specht basis of the coinvariant ring defined in \cite{ariki1997higher}. This construction coincides with the Gillespie-Rhoades higher Specht basis for $R_{n,k}$ \cite{RhoadesGillespie2021}. We give a proof for when $\lambda = (\lambda_1,\lambda_2)$ is a partition of two rows and $\ell(\lambda) = s$.
\end{abstract}

\section{Introduction}

The \emph{coinvariant algebra} is the quotient ring $R_n = \Q[\xx_n]/I_n$, where
$$\xx_n := \{x_1,\dots,x_n\}, \quad I_n = (e_1(\xx_n),\dots,e_n(\xx_n))$$
where $I_n$ can be thought of as the ideal of nonconstant symmetric functions. The study of $R_n$ and its various generalizations is a focal point in the rich and dynamic interplay between algebraic combinatorics, algebraic geometry, and representation theory. Borel proved \cite{borel1953cohomology} that $R_n \cong H^*(\mathcal{F}_n)$ is the cohomology ring of the type $A$ flag variety, where $x_i$ represents the first Chern class of the $i$th tautological line bundle.

On the other hand, there is an action of the symmetric group $\sym_n \circlearrowleft R_n$ defined by permuting the variables:
$$ \sigma \cdot f(x_1,\dots,x_n) = f(x_{\sigma(1)},\dots,x_{\sigma(n)})$$
which is a simple example of the \emph{Springer action}. Chevalley proved \cite{chevalley1955invariants}, as ungraded $\sym_n$-modules, that
$$ R_n \cong_{\sym_n} \Q[\sym_n]$$
and the graded character was computed by Lustzig (unpublished) and Stanley \cite{stanley1979invariants} to be
\begin{align}\label{eq: Frob Rn}
    \Frob_q(R_n) = \sum_{T \in \SYT_n} q^{\maj(T)} s_{\sh(T)}(\xx)
\end{align}
where $s_\lambda(\xx)$ is the \emph{Schur polynomial}.

The first of three generalizations of $R_n$ we consider is the \emph{Garsia-Procesi module} $R_\lambda$, where $\lambda \vdash n$ is a partition of $n$. Writing the transpose partition $\lambda' = (\lambda'_1 \geq \dots \geq \lambda'_n \geq 0)$ (allowing $\lambda'_i = 0$ to achieve a tuple of length $n$), the \emph{Tanisaki ideal} $I_\lambda$ is defined to be
$$ I_\lambda := ( e_d(S) : d > |S| - (\lambda'_n + \dots + \lambda'_{n-|S|+1}))$$
where $S \subset \{x_1,\dots,x_n\}$ and $e_d(S)$ is the degree $d$ partial elementary symmetric function in the variables $S$. The Garsia-Procesi module is defined to be $R_\lambda := \Q[\xx_n]/I_\lambda$. When $\lambda = (1^n)$ is a vertical strip, we have $R_{1^n} = R_n$. 
We have that $R_\lambda \cong H^*(\Sp_\lambda)$, where $\Sp_\lambda \subset \mathcal{F}_n$ is the Springer fiber corresponding to a nilpotent operator of Jordan type $\lambda$, due to work of DeConcini-Procesi \cite{deconcini_proc1981}. The presentation above was given by Tanisaki \cite{tanisaki1982ideals}. 
The graded character of $R_\lambda$ is the \emph{modified Hall-Littlewood polynomial} $\Frob_q(R_\lambda) = \Ht_\lambda(X;q)$ \cite{HottaSpringer1977Specialization}\cite{Springer1976GreenFunctions}\cite{garsia_procesi1992}.

The second generalization has to do with the famous \emph{Delta conjecture} in symmetric function theory. Haglund-Rhoades-Shimozono \cite{HaglundRhoadesShimozono2018} constructed a partial representation theoretic model for $\Delta_{e_{k-1}}' e_n$ (up to a twist) where $\Delta'$ is a certain operator on $q,t$-symmetric functions $\Lambda_{q,t}$, first defined by F. Bergeron-Garsia-Haiman-Tesler \cite{bergeron_garsia_haiman_tesler1999_identities}. The model depends on two parameters $k \leq n$, and
$$ I_{n,k} = (x_i^k : 1 \leq i \leq n) + (e_n(\xx_n),\dots,e_{n-k+1}(\xx_n)) \qquad R_{n,k} = \Q[\xx_n]/I_{n,k}$$
which specializes to the usual coinvariant algebra when $k = n$. They also computed the Frobenius character of $R_{n,k}$:
$$ \Frob_q(R_{n,k}) =  \rev_q \circ \,\omega\big(\Delta'_{e_{k-1}}e_n \big|_{t = 0}\big).$$
One presentation of the underlying geometry $R_{n,k}$ is given by the \emph{spanning line configurations} $X_{n,k} := \{ (\ell_1,\dots,\ell_n) \in (\C P^{k-1})^n: \ell_1 + \dots + \ell_n = \C^k$\} constructed by Pawlowski-Rhoades \cite{pawlowski_rhoades2019_flag}, who showed that
$ H^*(X_{n,k}) \cong R_{n,k}$.

The third generalization is an amalgamation of the previous two, and was due to the work of Griffin \cite{griffin2021ordered}. Define
$$ I\nls := (x_i^s: 1 \leq i \leq n) + (e_d(S):d > |S|-(\lambda_n'+\dots + \lambda'_{n-|S|+1})) \quad R\nls := \Q[\xx_n]/I\nls $$
where $\lambda \vdash k \leq n$, $\lambda' = (\lambda'_1 \geq \dots \geq \lambda'_n \geq 0)$ is padded with $0$'s to achieve a tuple of length $n$, and $s \geq \ell(\lambda)$. We refer to $R\nls$ as the \emph{$\Delta$-Springer module}, due to the specializations $R_{k,\lambda,\ell(\lambda)} = R_\lambda, R_{n,(1^k),k} = R_{n,k}$. Griffin-Levinson-Woo gave a geometric model \cite{griffin2024delta_springer} consisting of the \emph{Delta-Springer fibers} $Y_{n,\lambda,s}$ which consists of partial flags $F_\bullet \in \mathcal{F}_{1^n}(\C^K)$ preserved by a certain nilpotent operator depending on $\lambda$, and $K = s(n-k) + k$. They showed that $H^*(Y_{n,\lambda,s}) \cong R\nls$. These modules will be the main subject of our paper.

The Hilbert series of $R_n$ was first computed by E. Artin to be
$$\Hilb_q(R_n) = [n]_q! $$
where $[k]_q = 1 + q + \dots + q^{k-1}$, and $[n]_q! = [n]_q[n-1]_q\dots[1]_q$. There are multiple bases of $R_n$ indexed by permutations $\sigma \in \sym_n$ of combinatorial interest, of which we give a brief exposition.

The first basis is the \emph{Artin basis} of substaircase monomials $\AAA_n = \{\xx^\ax: \ax = (a_1,\dots,a_n), 0 \leq a_i \leq n-i\}$, which correspond to the \emph{inversion} statistic on permutations $\sigma \in \sym_n$. The Artin monomials are "geometric" in nature; there is a natural basis of $H^*(\mathcal{F}_n)$ given by the classes of cell closures $\overline{[C_w]}$ of $\mathcal{F}_n$ called the \emph{Schubert polynomials} $\sym_\sigma$. The leading term  of $\sym_\sigma$ with respect to the  lexicographical term order is exactly the Artin monomial $\xx^\ax$ corresponding to $\sigma$.

The three generalizations of $R_n$ also have monomial bases analogous to the Artin basis of $R_n$. 
Garsia-Procesi \cite{garsia_procesi1992} gave an inductive construction of a subset of Artin monomials which form a basis of $R_\lambda$. Haglund--Rhoades--Shimozono \cite{HaglundRhoadesShimozono2018} constructed a set of generalized Artin monomials $\AAA_{n,k}$ which form a basis of $R_{n,k}$, corresponding to an inversion statistic on ordered set partitions $\osp_{n,k}$. Griffin \cite{griffin2021ordered} generalized both of these constructions to form the $(n,\lambda,s)$-substaircase monomials, which form a monomial basis of $R\nls$.  

The next basis we consider, which is more representation-theoretic in nature, is the \emph{(Garsia-Stanton) descent basis}. Define the \emph{descent monomial} to be $g_\sigma(\xx) = \prod_{\sigma_i > \sigma_{i+1}} x_{\sigma_1}\dots x_{\sigma_i}$. Writing $g_\sigma(\xx) = \xx^\bx$, we have that $\bx = (b_1,\dots,b_n)$ has the property $b_1 + \dots + b_n = \maj(\sigma)$, the \emph{major index} of $\sigma$. This witnesses the equidistribution of $\inv$ and $\maj$, first proven by MacMahon \cite{macmahon1913indices}: 
$$\sum_{\sigma \in \sym_n} q^{\maj(\sigma)} = \sum_{\sigma \in \sym_n} q^{\inv(\sigma)}.$$
They were first proven to be a basis by Garsia \cite{garsia1980cmbc} using Stanley-Reisner theory, and later studied by Garsia-Stanton in \cite{garsia_stanton1984}. Straightening algorithm proofs were given by Adin-Brenti-Roichman \cite{adin_brenti_roichman2005} and E.E. Allen \cite{AllenDescent}. Haglund--Rhoades--Shimozono construct a generalized Garsia-Stanton basis $\mathcal{GS}_{n,k}$ corresponding to a $\maj$ statistic on ordered set partitions $\osp_{n,k}$.

Carlsson and the first author \cite{carlsson2024descent} found a subset of the Garsia-Stanton monomials which descends to a basis of the Garsia-Procesi module $R_\lambda$. The second author \cite{hanada2025charge} reformulated the construction in terms of the Lascoux-Sch\"{u}tzenburger \emph{cocharge} and \emph{catabolizability}, which directly connects this basis with the graded Frobenius character of $R_\lambda$.

The final basis we consider is the \emph{higher Specht basis} of Ariki-Terasoma-Yamada \cite{ariki1997higher}. They are generalizations of the usual Specht polynomial for a standard Young Tableau $T$ given by the following:
$$ F_T = \prod_{T_i \text{ is a column of } T} \prod_{j,k \in T_i} (x_j-x_k).$$
Ariki-Terasoma-Yamada define the \emph{higher Specht polynomial} for two standard Young tableaux $S,T$ of the same shape:
$$ F_T^S := \varepsilon_T \cdot \xx_T^{\cc(S)}, \qquad \varepsilon_T = \sum_{\tau \in C(T)}\sum_{\sigma \in R(T)} (-1)^\tau \tau\sigma \in \Q[\sym_n]$$
where $C(T),R(T)$ are the column and row permutation groups respectively, and $\cc(S)$ is the \emph{cocharge word} of the tableau $S$ defined in Section \ref{sec:bg}.
The advantage of the higher Specht basis is that for a fixed $S$, the polynomials $\{F_T^S: T \in \SYT_n, \sh(S) = \sh(T)\}$, span a copy of the irreducible $\sym_n$-representation $V_{\sh(S)}$. Since $\{F_T^S: S,T \in \SYT_n, \sh(S) = \sh(T)\}$ forms a basis of $R_n$ \cite{ariki1997higher}, this set explicitly realizes the decomposition
$$ R_n \cong \bigoplus_{\lambda} V_\lambda^{\oplus |\SYT(\lambda)| }$$
as a graded $\sym_n$-module.
There is a generalization of the higher Specht basis for $R_{n,k}$, as well as a conjectured one for $R_\lambda$, due to Gillespie-Rhoades \cite{RhoadesGillespie2021}. 

In this paper, we construct two different bases of $R\nls$.
Our first basis is an extension of the Garsia-Stanton basis.
Let $\lambda \vdash k \leq n$ be a partition, and $s \geq \ell(\lambda)$. Let $\DD_n =\{\bx : \xx^\bx = g_\sigma(\xx) \text{ for some } \sigma \in \sym_n\}$, and let $\osp_{n,\lambda'} := \{ \sigma = (\AAAA | \BBB): \AAAA \in \osp_{\lambda'}, \BBB \in \osp_{n-k,n-k}\}$. Define
\begin{align*}
    \mathcal{D}\nls: = \{\ax : \ax\vert_{A_i} \in \mathcal{D}_{\lambda'_i}, \ax\vert_{B_j} \in \{0,\dots, s-1\}  \text{ for some } \sigma \in \osp_{n,\lambda'}\}.
\end{align*}

Note that we can also reformulate this in terms of cocharge. Define the set $\Sigma\nls := \{(S,(\mu): S \in \SYT_n, \ctype(S|_k) \trianglerighteq \lambda, \des(S) < s, \mu \subseteq (n-k)\times(s-\des(S)-1)\}$. This is the collection of standard Young tableaux whose restriction to the first $k$ entries satisfies a certain \emph{catabolizability} condition depending on $\lambda$, along with a partition $\mu$ of at most $(n-k)$ parts with $\mu_1 < s-\des(S)$.
We can realize each $\bx\in \mathcal{D}\nls$ as being indexed by pairs $(w,\mu)$ such that $(P(w),\mu)\in \Sigma\nls$.
Using this correspondence, we show our first theorem.
\begin{thma}\makeatletter\def\@currentlabel{A}\makeatother\label{thm: A}
    The set of monomials $\BB\nls := \{ \xx^\bx : \bx \in \DD\nls\}$ descend to a monomial basis of $R\nls$.
\end{thma}

Furthermore, we can use this construction to obtain a new combinatorial formula for the Schur expansion of the graded Frobenius character $\Frob_q(R\nls)$.
\begin{thmb}\makeatletter\def\@currentlabel{B}\label{thm: b}
We have 
\begin{align}
    \Frob_q(R\nls) = \sum\limits_{(S,\mu)\in \Sigma\nls} q^{\cocharge(S) + |\mu|} s_{\shape(S)}.
\end{align}
\end{thmb}
This construction is naturally connected to the theory of \emph{battery-powered tableaux}, a combinatorial object originally introduced by Gillespie--Griffin \cite{battery}. Gillespie--Griffin have a Schur expansion formula for $\Frob_q(R\nls)$ in terms of battery-powered tableaux. Though the formula is combinatorial, the motivation and proof of it were geometric. Furthermore, they were unable to connect the formula back to the known combinatorics of $R\nls$. Our work bridges this gap: in particular, we show that Theorem \ref{thm: b} is equivalent to their formula through a cocharge/shape preserving bijection between $\Sigma\nls$ and the set of corresponding battery-powered tableaux. 

\begin{corc}
    There is a weight-preserving bijection between the sets $\Sigma\nls$ and the battery-powered tableaux $\mathcal{T}^+\nls$.
\end{corc}

Our second basis is a higher Specht basis for $R\nls$ for $\lambda$ with at most two rows; as a consequence, we obtain a higher Specht basis for $R_\lambda$ for two row partitions. Note that we identify a partition $\mu \subseteq (n-k)\times (s-\des(S)-1)$ with a tuple $(i_1,\dots, i_{n-k})$ of nonnegative integers that sum to at most $(s-\des(S)-1)$ (for details, see Section \ref{subsec: higher specht background}).

\begin{thmc}
    Let $\lambda = (\lambda_1,\lambda_2)$ be a partition of $k$ with two rows where $k \leq n$.  The set of polynomials
    $$ \CC_{n,\lambda,\ell(\lambda)}:= \{ F_T^S e_1^{i_1}\dots e_{n-k}^{i_{n-k}} : (S,(i_1,\dots,i_{n-k})) \in \Sigma_{n,\lambda,\ell(\lambda)},\sh(S) = \sh(T)\}$$
    forms a higher Specht basis of $R_{n,\lambda,\ell(\lambda)}$.
\end{thmc}

We conjecture the theorem to hold for general $(n,\lambda,s)$:

\begin{conja}
    The set of polynomials $\CC\nls$ descends to a higher Specht basis of $R\nls$ for arbitrary $\lambda \vdash k \leq n$, $s\geq \ell(\lambda)$.
\end{conja}

The paper is organized as follows. In \textbf{Section} \ref{sec:bg} we give definitions and background on combinatorial preliminaries, cocharge, catabolizability, higher Specht bases, and the module $R\nls$. 
In \textbf{Section} \ref{sec: tab constuction for R_{n,k}}, we reframe existing results of Haglund--Rhoades--Shimozono using the language of tableaux and cocharge.
In \textbf{Section} \ref{sec:counting}, we generalize the framework introduced in the previous section for $R\nls$. Using this, we establish a few fundamental theorems about the cardinalities of our indexing sets for our bases. 
In \textbf{Section} \ref{sec:des} we prove that the set $\BB\nls$ is a monomial basis of $R\nls$. In \textbf{Section} \ref{sec: Frob char}, we use the descent basis to give a combinatorial formula for the graded Frobenius character of $R\nls$. We show that it coincides with the battery-powered tableaux formula of Gillespie--Griffin. In \textbf{Section} \ref{sec:specht}, we construct a conjectured higher Specht basis for $R\nls$ and prove it for the case where $\lambda$ is a two row partition. We also compare our conjectured formula with that of Gillespie-Rhoades \cite{RhoadesGillespie2021}.

\subsection{Acknowledgements}

The authors are grateful to Brendon Rhoades, Nantel Bergeron, Sean Griffin, and Maria Gillespie for helpful conversations. 


\section{Background}\label{sec:bg}
\subsection{Partitions, Permutations, and Tableaux}
A \textit{composition} $\alpha$ of $n$, denoted $\alpha\models n$, is a sequence $(\alpha_1,\alpha_2,\dots,\alpha_l)$ where $\alpha_i >0$ and $\alpha_1 +\alpha_2 + \cdots+\alpha_l = n$.
A \textit{partition} $\lambda = (\lambda_1,\lambda_2,\dots,\lambda_{\ell(\lambda)})$ of $n$, denoted $\lambda\vdash n$, is a composition satisfying $\lambda_1\geq \lambda_2\geq \dots \geq \lambda_{\ell(\lambda)}>0$. We denote the transpose of $\lambda$ as $\lambda'$.
Throughout, we write Young diagrams of partitions in French notation, meaning that the first part corresponds to the bottom row. We set $n(\lambda):= \sum_{i = 1}^{\ell(\lambda)} (i-1)\lambda_i = \sum_{i = 1}^{\ell(\lambda')} \binom{\lambda_i'}{2}$.
 
There is a partial ordering on the set of all the partitions of size $n$ called the \textit{dominance} ordering, denoted $\unrhd$, defined by 
\[\mu\unrhd \lambda \ \  \Leftrightarrow \ \  \mu_1 + \cdots \mu_k \geq \lambda_1+\cdots  \lambda_k \text{ for all } k\]
It is well known that $\mu \unrhd \lambda \ \  \Leftrightarrow \ \   \mu' \unlhd \lambda'$. If we move a box of $\lambda$ to a lower row so that the resulting shape $\mu$ is a partition, we have $\mu \unrhd \lambda.$
Dominance is the transitive closure of moving boxes down. 

A \textit{semistandard (Young) tableau} of shape $\lambda$ is a filling of $\lambda$ such that the rows are weakly increasing from left to right and the columns are strictly increasing from bottom to top. The \textit{weight} of a semistandard tableau $T$ is the tuple $(m_1, m_2,\dots)$  where $m_i$ is the number of times $i$ appears in $T$. The reading word $\rw(T)$ of a tableau $T$ is the word we get by concatenating the row words from top to bottom.
We denote the shape of $T$ by $\shape(T)$. Let $\SSYT(\mu)$ denote the set of semistandard tableaux of weight $\mu$ and $\SSYT(\lambda,\mu)$ denote the set of semistandard tableaux of shape $\lambda$, weight $\mu$. The \emph{Kostka number} $K_{\lambda,\mu}$ counts such tableaux: we have that $K_{\lambda,\mu} = |\SSYT(\lambda,\mu)|$.  

A \textit{standard (Young) tableau} of size $n$ is a semistandard tableau of weight $(1^n$). For a given shape, we denote the set of standard Young tableaux of shape $\lambda$ by $\SYT(\lambda)$. We denote the set of all standard Young tableaux of partition shape and size $n$ by $\SYT_n$.

We can describe the Kostka numbers $K_{\lambda,\mu}$ using standard tableaux in the following way:
 \begin{prop}\label{prop: SYT and Des}
     For two partitions $\lambda,\mu = (\mu_1,\mu_2,\dots, \mu_l)\vdash n$, we have
      \begin{align}\label{eq: kostka numb}
     K_{\lambda,\mu} = | \{T\in \SYT(\lambda)  \ | \ \Des(T) \subset \{\mu_1,\mu_1 + \mu_2, \dots, \mu_1 +\cdots + \mu_{l-1}\} \} |
 \end{align}
 \end{prop}
\noindent For a proof, see \cite[Proposition 2.3]{hanada2025charge}.

The Schensted correspondence gives a bijection from permutations $w\in \mathfrak{S}_n$ to pairs $(P(w),Q(w))$ of standard Young tableau of size $n$ of the same shape. We say $P(w)$ is the \textit{insertion tableau} and $Q(w)$ is the \textit{recording tableau} of $w$.
Two permutations $w,w'$ are \textit{Knuth equivalent} if $P(w) = P(w')$. We also know that for a fixed $S\in \SYT(\lambda)$:
$|\{w\in \sym_n | P(w) = S\}| = |\SYT(\lambda)| = K_{\lambda,1^n}$.

Throughout, we write permutations $w = w_1\dots w_n \in \mathfrak{S}_n$ in one line notation. For any word $\mathbf{z}$, let $\rev(\mathbf{z})$ denote the reverse word of $\mathbf{z}$. For any subset $U\subset [n] = \{1,\dots, n\}$, let $w\vert_{U}$ denote the subword of $w$ consisting of $w_i$ such that $w_i\in U$.
Similarly, denote the restriction of a standard Young tableau $S$ to the entries $[k]$ by $S\vert_{k}$. Note that $S\vert_{k} \in \SYT_k$. 

We state a few properties of the Schensted correspondence without proof. For a detailed reference, see \cite[Chapter 7]{EC2}.
\begin{prop}\label{prop: RSK props}
The following properties hold:
    \begin{enumerate}[label=(\alph*)]
        \item For any $w\in \sym_n$ and $k\leq n$, we have that $P(w\vert_k) = P(w)\vert_k$. (\cite[Corollary of Lemma 7.11.2]{EC2}) \label{lem: RSK restriction}
        \item  For any $w\in S_n$, we have $\Des(w) = \Des(Q(w))$. (\cite[Lemma 7.23.11]{EC2})\label{prop: RSK des}
        \item For any $w\in S_n$, we have $P(\rev(w))= (P(w))^t.$(\cite[Corollary A1.2.11]{EC2})\label{thm:RSK rev}
    \end{enumerate}
\end{prop}

We can also rewrite Proposition \ref{prop: SYT and Des} using some fixed $S\in \SYT(\lambda)$ in the following way:
\begin{cor}\label{cor: kostka number permutations}
 For two partitions $\lambda,\mu = (\mu_1,\mu_2,\dots, \mu_l)\vdash n$ and some fixed $S\in \SYT(\lambda)$, we have
    \begin{align}
      K_{\lambda,\mu} = | \{w \in \sym_n , P(w) = S\ | \ \Des(w) \subset \{\mu_1,\mu_1 + \mu_2, \dots, \mu_1 +\cdots + \mu_{l-1}\} \} |.
\end{align}
\end{cor}
\begin{proof}
    From Proposition \ref{prop: RSK des}, we have a bijection between the $T\in \SYT(\shape(S))$ with $\Des(T) \subset \{\mu_1,\mu_1 + \mu_2, \dots, \mu_1 +\cdots + \mu_{l-1}\} \}$ and permutations $w\in \sym_n$ satisfying $P(w) = S, \Des(w)\subset  \{\mu_1,\mu_1 + \mu_2, \dots, \mu_1 +\cdots + \mu_{l-1}\} \}$ which sends $T$ to the permutation $w$ with $P(w) = S$, $Q(w) = T$.
\end{proof}

\subsection{Symmetric Functions and Frobenius Character}

We fix our notations for symmetric functions here. We refer the reader to \cite{Macdonald1995Symmetric} for a detailed treatment. Let $\xx = \{x_1,x_2,\dots \}$ be an infinite set of variables, and let $\Lambda_F$ denote the \emph{ring of symmetric functions} in these variables with conefficients in a field $F$. We typically take $F = \Q$ or $F = \Q(q)$. We denote by $e_\lambda(\xx), h_\lambda(\xx), s_\lambda(\xx), m_\lambda(\xx)$ the \emph{elementary}, \emph{complete homogeneous}, \emph{Schur}, and \emph{monomial} symmetric functions respectively. We omit the $\xx$-variables when it does not cause confusion. All four of these are bases of $\Lambda_F$, meaning we can rewrite any symmetric function $f$ using one of the four families. In particular, we say a  symmetric function $f$ is \textit{Schur-positive} if all the coefficients in the Schur expansion of $f$ are nonnegative integers. 
Note that the Kostka numbers appear as the coefficients of the Schur expansion of $h_\mu$: that is, we have $ h_\mu  = \sum_{\lambda\vdash|\mu|} K_{\lambda,\mu} s_{\lambda}.$

The \textit{Hall inner product} is the symmetric inner product defined on $\Lambda$ by the relation $\langle m_\lambda, h_\gamma\rangle = \delta_{\lambda,\gamma}$, where $m_\lambda$ is the \emph{monomial} symmetric function. 
The Schur functions are an orthonormal basis with respect to this inner product.
We also have an involution $\omega$ on $\Lambda$, defined by $\omega e_\lambda= h_\lambda$. The map $\omega$ is an isometry: that is, we have $\langle f,g\rangle = \langle \omega f,\omega g\rangle$ for any $f,g \in \Lambda$.
We also have $\omega s_\lambda = s_{\lambda^t}$ for any $\lambda$.

At times, we will look at elementary symmetric polynomials in some finite set of variables $S \subset \xx$. We define
$$ e_m(S) := \sum_{\{i_1 < \dots < i_m\} \subset S} x_{i_1}\dots x_{i_m}$$
to be the sum of all degree $m$ squarefree monomials with variable indices in $S$.

Throughout, we will study $q$-symmetric functions, where we take $F = \mathbb{Q}(q)$.
The \emph{(transformed) Hall-Littlewood polynomial} is 
$$ H_\mu(X;q) = \sum_{\lambda \vdash n} K_{\lambda,\mu}(q) s_\lambda(\xx)$$
where $K_{\lambda,\mu}(q)$ is the Kostka-Foulkes polynomial. Note that for $q=1$, we have that $K_{\lambda,\mu}(1)$ is just the Kostka number $K_{\lambda,\mu}$. If we reverse, the grading, the result is the \emph{modified} Hall-Littlewood polynomials, which are more combinatorial in nature:
\begin{align}\label{eq: modified HL}
    \Ht_\mu(X;q) = q^{n(\mu)}H(X;q^{-1}) = \sum_{\lambda \vdash n} \widetilde{K}_{\lambda,\mu}(q)s_\lambda(\xx)
\end{align}
where $\widetilde{K}_{\lambda,\mu}(q) = q^{n(\mu)}K_{\lambda,\mu}(q^{-1}) = \sum_{T \in \SSYT(\lambda,\mu)} q^{\cocharge(T)}$ is the \emph{modified Kostka-Foulkes polynomials}. We have that $H_\mu[X,1] = \Ht_\mu[X;q] = h_\mu$. For more details, see \cite{Macdonald1995Symmetric}

Symmetric functions are deeply connected with the study of $\sym_n$-representations through the Frobenius characteristic map.
Recall that the irreducible $\sym_n$-representations $V_\lambda$ are indexed by partitions $\lambda \vdash n$. We refer to $V_\lambda$ as the \emph{Specht module}. Given a finite-dimensional $\sym_n$-module $V$, by complete reducibility we may write $V$ in terms of irreducibles:
$$ V \cong \bigoplus_\lambda V_\lambda^{\oplus c_\lambda}.$$
The ungraded \emph{Frobenius character} of $V$ is a symmetric function defined by $\Frob(V) = \sum_\lambda c_\lambda s_\lambda$, where $s_\lambda$ is the Schur function associated with $\lambda$. Note that since $c_\lambda$ counts the multiplicity of $V_\lambda$ in $V$, we have that $c_\lambda \in \N$. In the case where $V= \oplus_{d\geq 0} V_d$ is a graded $\sym_n$-module, we can define the \emph{graded Frobenius character} to be $\Frob_q(V) = \sum_{d\geq 0} \Frob(V_d)$.

For any $\sym_n$-module $V$, we can recover $\dim(V)$ from the Schur expansion of $\Frob(V)$ by replacing each $s_\lambda$ with $\dim(V_{\lambda}) = K_{\lambda,1^n}$. Conversely, we can study the dimension of certain subspaces of $V$.
For the Young subgroup $S_\gamma = S_{\gamma_1} \times \cdots \times S_{\gamma_l}\subset S_n$ corresponding to $\gamma\models n$, we define $N_{\gamma} = \sum_{\sigma\in S_\gamma} \sgn(\sigma) \sigma$ to be the antisymmetrizer with respect to $\gamma$. 
The vector space $N_\gamma V$ is the subspace of elements of $V$ that are antisymmetric with respect to $S_\gamma$.
Let $\mathcal{E}\uparrow^{S_n}_{S_\gamma}$ denote the induction of the sign representation $\mathcal{E}$ of $S_\gamma$ to $S_n$. It is well known that $\Frob(\mathcal{E}\uparrow^{S_n}_{S_\gamma}) = e_\gamma$. Using Frobenius reciprocity (see \cite[Chapter 3.3]{Fulton-Harris}), we get the following result:

\begin{prop}\label{prop: dim and antisym}
    For any $\sym_n$-representation $V$ and Young subgroup $S_\gamma \subset S_n$, we have  $$\dim(N_\gamma V) = \langle e_\gamma, \Frob(V)\rangle.$$
\end{prop}
The analogous statement for graded modules holds as well, if we replace $\dim(N_\gamma V)$ (resp. $\Frob(V)$) with  $\Hilb_q(N_\gamma V)$ (resp. $\Frob_q(V)$).
Since we know that any homogeneous symmetric function $f$ of degree $n$ is uniquely determined by the values $\langle e_\gamma, f\rangle$ for all $\gamma\vdash n$, we see that knowing  $\Hilb_q(N_\gamma V)$ for all $\gamma\models n$ uniquely defines $\Frob_q(V)$.

\subsection{Descent Words and Descent Bases}

The Artin basis of $R_n$ arises naturally from Gr\"{o}bner theory, as they are the nonleading terms with respect to the lexicographical order. The Garsia-Stanton descent basis arises from considering a different order, which we define here.

\begin{defn}\label{def: descent-order}
    Let $\ax = (a_1,\dots,a_n) \in \Z_{\geq 0}^n$. Define $\sort(\ax) := (a_{i_1} \geq \dots \geq a_{i_n})$ to be the unique way to write the entries of $\ax$ in weakly decreasing order. Then, we say that $\ax <_{\des} \bx$ if:
    \begin{align*}
        \begin{cases}
            \sort(\ax) <_{\lex} \sort(\bx) \text{    or,}\\
            \sort(\ax) = \sort(\bx) \text{ and }\ax <_{\lex} \bx
        \end{cases}
    \end{align*}
\end{defn}

Note that $\des$ does \emph{not} induce a monomial order on $\Q[\xx_n]$, as we do not have $\ax <_{\des} \bx \implies \ax + \cx <_{\des} \bx + \cx$, but is only a total order on monomials of $\Q[\xx_n]$. The descent order, however, satisfies the following useful lemma.

\begin{lem}\label{lem: descent-order-subset}
    Let $A \subset [n]$ with complement $A'$. If $\alpha$ is a composition, denote $\alpha\vert_A := (\alpha_{i_1},\dots,\alpha_{i_s})$, where $A = \{i_1 < \dots < i_s\}$. Then, we have
    $$ \alpha\vert_A \leq_{\des} \beta\vert_A, \quad \alpha\vert_{A'} \leq_{\des} \beta\vert_{A'} \implies \alpha \leq_{\des} \beta$$
\end{lem}

For a permutation $\sigma\in \mathfrak{S}_n$, let $g_\sigma(\xx)$ be the corresponding descent monomial, defined to be $g_\sigma(\xx) := \prod_{\sigma_i > \sigma_{i+1}} x_{\sigma_1}\dots x_{\sigma_i}.$
 Let $\mathcal{D}_n = \{ \ax \ | \ \xx^{\ax} = g_{\sigma}(\xx) \text{ for } \sigma\in \mathfrak{S}_n\}.$ We refer to these exponents as \textit{descent words}.
 The descent monomials $\{g_\sigma(\xx) \  | \ \sigma\in \mathfrak{S}_n\}$ are a vector space basis of the coinvariant ring $R_n$ \cite{garsia1980cmbc}\cite{garsia_stanton1984}\cite{Steinberg1975Pittie}. The descent monomials can be characterized as being the nonleading monomials of $R_n$ with respect to the descent order.

We have analogous bases for  $R_{n,k}$ and $R_\lambda$. 
For $R_{n,k}$, the following basis is due to Haglund--Rhoades--Shimozono \cite{HaglundRhoadesShimozono2018}.
 
\begin{thm}[Haglund--Rhoades--Shimozono \cite{HaglundRhoadesShimozono2018}]\label{thm: HRS descent}
The following set $\mathcal{GS}_{n,k}$ descends to a basis of $R_{n,k}$. Furthermore, they are the non-leading terms with respect to the $\des$-order.
\begin{align}
    \mathcal{GS}_{n,k} = \left\{g_\sigma(\xx) \prod\limits_{j=1}^{n-k} (x_{\sigma_1}\dots x_{\sigma_{j}})^{i_j} \ | \ \sigma\in \mathfrak{S}_n, 0\leq i_1+\dots i_{n-k}< k-\des(\sigma) \right\}
\end{align}. 
\end{thm}
For the Garsia-Procesi rings $R_{\mu}$, the following formula is from \cite{carlsson2024descent}.
For any partition $\mu\vdash n$, define $\osp_\mu$ to be the collection of ordered set partitions $(A_1| \dots | A_{\ell(\mu)})$ of $[n]$ where $|A_i| = \mu_i$.
Using this, we can define the following set:
\begin{align}
    \mathcal{D}_\mu: = \{\ax : \ax\vert_{A_i} \in \mathcal{D}_{\mu'_i} \text{ for } (A_1| \dots | A_{\mu_1})\in \osp_{\mu'}\}.
\end{align}

Note that we take $\osp_{\mu'}$, so the sizes of the parts correspond to the lengths of the columns.
We have that $\mathcal{D}_\mu \subset \mathcal{D}_n$, hence  $\{\xx^\ax \ | \ \ax \in \mathcal{D}_\mu\}\subset \{g_\sigma(\xx) \ | \ \sigma\in \mathfrak{S}_n\}$. Furthermore, we have that this subset is in in fact a basis of $R_\mu$.

\begin{thm}[Carlsson--C.\cite{carlsson2024descent}]
 The collection of monomials $\{\xx^\ax \ | \ \ax \in \mathcal{D}_\mu\}$ descends to a basis of $R_\mu$.
    Furthermore, they are the non-leading terms with respect to the $\des$-order on monomials.
\end{thm}


\subsection{Cocharge}
We recall the definitions for permutations and standard tableaux. 

\begin{defn}\label{def: cc word}
    The \textit{cocharge word} of $w\in \mathfrak{S}_n$ (denoted $\cc(w)$) is a word of length $n$ consisting of the labeling of $w$ that we obtain in the following way. Label the 1 of $w$ with $0$. Assume we labeled the letter $i$ with $k$. 
    Label $(i+1)$ with a $k$ if it is to the right of $i$. We label $(i+1)$ with a $(k+1)$ if it is to the left of $i$. 
\end{defn}

This process is equivalent to scanning $w$ from left to right, labeling entries in increasing order.
The label of $i$ corresponds to the number of times we need to return to the left end of $w$ before we label $i$.

\begin{defn}
    The statistic $\textit{cocharge}(w)$ is the sum of the letters in $\cc(w)$. 
\end{defn}

\begin{example}\label{ex: cocharge word}
 Let $w = 3 \ \ 5 \ \ 1 \ \ 6\ \ 2 \ \ 4 \ \ 7 .$  
The corresponding cocharge word is $\cc(w)  = 1  \ \ 2 \ \ 0 \ \ 2 \ \ 0 \ \ 1 \ \ 2$, hence $\text{cocharge}(w) = 1 + 2 + 0 + 2 + 0 + 1 + 2 = 8.$
\end{example}

Using the algorithm to construct $\cc(w)$, we can find the following criteria for a word $\mathbf{z}$ to be a valid cocharge word of a permutation.
\begin{lem}\label{lem: classification of cocharge words}
        Let $\mathbf{z}$ be a word of nonnegative integers of length $n$ containing a 0.
        We have $\mathbf{z} = \cc(w)$ for some $w\in \mathfrak{S}_n$ if and only if for any nonnegative $c$, one of the following holds:
    \begin{enumerate}[label=(\alph*)]
        \item there exists a $(c+1)$ to the left of the rightmost $c$ in $\mathbf{z}$.
        \item $c$ is the largest letter in $\mathbf{z}$.
    \end{enumerate}
\end{lem}
\begin{proof}
    It is clear that any cocharge word $\cc(w)$ satisfies the condition above, since the cocharge goes up whenever we encounter an $(i+1)$ in $w$ such that $(i+1)$ to the left of $i$.
    
    We can recover the permutation $w$ from  $\mathbf{z}$ by reversing the process. That is: we can construct an ordered set partition $(A_0|\dots |A_k)$ of $n$, where $\mathbf{z}\vert_{A_i} =  \underbracket{i\dots i}_{|A_i| \text{ many}}$.

    Define $w$ so that $w\vert_{A_i} = (a_i +1)\dots a_j$ where $a_i = \sum_{j=0}^i |A_j|$.
    The resulting word $w$ is a permutation by construction. The fact that $\cc(w) = \mathbf{z}$ follows from the fact that the cocharge goes up when $a+1$ appears to the left of $a$: from (a), we know that $(a_j + 1)$ always appears to the left of $a_j$, hence the cocharge value will go up. 
\end{proof}

We can modify cocharge words, by inserting or removing certain letters, without changing the fact that it is a cocharge word. 
\begin{lem}\label{lem: insert, delete}
    Let $\mathbf{z} = \cc(w)$ for $w\in \mathfrak{S}_n$.
    \begin{enumerate}[label=(\alph*)]
        \item  For a nonnegative integer $d$ such that $d = z_i$ for some $i$ and any word $\mathbf{u}$ of length $n+1$ such that $\mathbf{u}_j = d$, $\mathbf{u}\vert_{[n+1]-\{j\}} = \mathbf{z}$, we have $\mathbf{u}= \cc(\tilde{w})$ for some $\tilde{w}\in \sym_{n+1}$.
        \item For index $m$ such that $z_m$ is the rightmost, largest entry in $\mathbf{z}$,  we have $\rev(\mathbf{z}\vert_{[n]-\{m\}})\in \mathcal{D}_{n-1}$.  In particular,  we have  $\mathbf{z}\vert_{[n]-\{m\}} = \cc(w\vert_{n-1})$.
    \end{enumerate}
\end{lem}

\begin{proof}
    \begin{enumerate}[label=(\alph*)]
        \item Note that such $\mathbf{u}$ is equivalent to inserting $d$ into $\mathbf{z}$. Any word we obtain in this way still satisfies the conditions in Lemma \ref{lem: classification of cocharge words} since $\mathbf{z}$ does.
        
        \item Since we label the entries by reading from left to right, we know that the rightmost, largest entry in $\mathbf{z}$ must be the last entry in $w$ that we label. Thus $w_m = n$. Thus deleting $z_m$ from $\mathbf{z}$ is equivalent to labeling all but the largest entry in $w$.
    \end{enumerate}
\end{proof}
Now, we extend the definitions to standard tableaux. 
Define $\cocharge(S) := \cocharge(\rw(S))$ for any standard tableau $S$.
Analogous to how we construct the cocharge word of a permutation, we can construct a cocharge tableau $\cc(S)$ for each standard tableau $S$ by replacing $i$ in $S$ with the corresponding cocharge label of $i$ in $\rw(S)$. Let $\cc(S)_i$ denote the entry in $\cc(S)$ corresponding to $i$ in $S$. 

Alternatively, we can define $\cc(S)$ directly on tableaux:
\begin{defn}\label{def: cc(S)}
    The \textit{cocharge tableau} of a standard tableau $S$ (denoted $\cc(S)$) is a filling of the same shape consisting of the labeling of $S$ that we obtain in the following way. Label the 1 of $S$ with $0$. Assume we labeled the letter $i$ with $k$. 
    Label $(i+1)$ with a $k$ if is in the same row or lower than $i$. We label $(i+1)$ with a $(k+1)$ if it is in a higher row than $i$.
\end{defn}
It is a simple exercise to check that these two definitions are equivalent.
Note that for Definition \ref{def: cc(S)}, we do not specify that $S$ must be of partition shape. 

 \ytableausetup{smalltableaux}
\begin{example}
Let $\ytableaushort{6,35,1247}$. Then $\cc(S) = \ytableaushort{3,12,0013}$.
In this case, $\cc(S)_4 = 1$.
\end{example}

We record the following observation, which follows immediately from Definition \ref{def: cc(S)}.
\begin{lem}\label{lem: max cc}
    For any $S\in \SYT_n$, we have $\cc(S)_i\leq \des(S)$ for any i.
\end{lem}

We also have the following lemma, which is the tableaux version of Lemma \ref{lem: classification of cocharge words}. 
\begin{lem}\label{lem: cc(S) condition}
   Let $Z$ be a semistandard tableaux of size $n$ consisting of nonnegative integers and contains a 0. We have $Z = \cc(S)$ for some $S\in \SYT_n$ if and only if for any nonnegative integer $c$, one of the following holds:
   \begin{enumerate}[label=(\alph*)]
       \item $c$ is the largest letter in $Z$
       \item There exists a $(c+1)$ in some row higher than the row containing the lowest $i$.
   \end{enumerate}
\end{lem}

Note that Lemma \ref{lem: cc(S) condition} does not depend on whether the shape of $Z$ is partition or skew shape. 

We can write $\Frob_q(R_n)$ using the cocharge statistic on standard tableaux: 
\begin{align}
    \Frob_q(R_n) = \sum\limits_{S\in \SYT_n} q^{\cocharge(S)} s_{\shape(S)}.
\end{align}
This is equivalent to \eqref{eq: Frob Rn} using the fact that maj and cocharge are equidistributed amongst SYT \cite{Killpatrick}.
However, we have a stronger version of this equidistribution. This was a folklore result, with a recent proof due to Chan \cite{kelvinthesis}.
\begin{prop}[Chan \cite{kelvinthesis}]\label{prop: maj, cc}
For any partition $\mu$, we have 
\[\sum\limits_{S\in \SYT(\mu)} q^{\maj(S)} t^{\des(S)} =\sum\limits_{S\in \SYT(\mu)} q^{\cocharge(S)} t^{\des(S)}. \]
\end{prop}

Finally, we note of the following result, due to Lascoux-Sch\"{u}tzenberger:
\begin{thm}[Lascoux-Sch\"{u}tzenberger \cite{LS1}]
    If $w,w' \in \mathfrak{S}_n$ are such that $P(w) = P(w')$, then $\cocharge(w) = \cocharge(w')$.
\end{thm} 

From this, we know that cocharge is constant on Knuth equivalence classes. That is: for any $w\in \mathfrak{S}_n$ such that $P(w) = S$, we have $\cocharge(w) = \cocharge(S)$.

We can think of a descent monomial as being a  ``reversed cocharge monomial", using the fact that $g_\sigma(\xx) = \xx^{\rev(\cc(w))}$ where $w= \text{rev}(\text{rev}(\sigma)^{-1})$.
\begin{prop}\label{prop: descent is rev cc}
    $\mathcal{D}_n = \{\rev(\cc(w)) \ | \ w\in \mathfrak{S}_n\}$.
\end{prop}
Throughout this paper, we write descent monomials using cocharge words, since this groups the monomials nicely based on insertion tableaux.

\subsection{Lascoux Standardization and Catabolizability type}
For $\lambda \vdash n$, we have that $\Frob_q(R_\lambda)$ is given by the modified Hall--Littlewood polynomial $\widetilde{H}_\lambda[X;q]$ \eqref{eq: modified HL} which can be expressed in terms of semistandard tableaux of weight $\lambda$.
However, in order to better understand the structure of $R_\lambda$ as a quotient of $R_n$, it is beneficial to write $\Frob_q(R_\lambda)$ as a sum over a subset of $\SYT_n$. We can identify the appropriate subset using catabolizability types of standard tableaux.

The \textit{catabolizability type} $\ctype(S)$ of $S\in \SYT_n$ is a partition of size $n$ defined using an operation called catabolism on tableaux. 
For more on catabolism, see \cite{Lascoux}, \cite{SW}, or \cite{hanada2025charge}. Here, we omit the precise definitions of catabolisms and catabolizability type and instead focus on the properties we use in the later sections of our work. 





In general, many things are unknown about catabolism and catabolizability, though the operation of catabolism itself is very easy to compute.
However, there is a concrete algorithm, due to Blasiak \cite{Blasiak}, called the \textit{catabolism insertion algorithm}, which computes the catabolizability type of a permutation $w$, where we define $\ctype(w) : = \ctype(P(w))$.

\begin{algo}[Blasiak {\cite[Algorithm 3.2]{Blasiak}}]\label{alg: ctype}
We define a function $f$ on pairs $(x,\nu)$, where $x = ya$ is a word ($a$ is the last letter of $x$) and $\nu$ is a partition, to be
\begin{align*}
    f(x,\nu) &= \begin{cases}
        (y,\nu + \epsilon_{a+1}) & \text{if }\nu + \epsilon_{a+1} \text{ is a partition,}\\ 
        ((a+1)y,\nu) & \text{otherwise.}
    \end{cases}
\end{align*}
where $\epsilon_{a+1} $ is the composition $(0,\dots,0, 1,0,\dots)$ with 1 in the $(a+1)$th coordinate.
 
Let $w$ be a permutation of length $n$ with cocharge word $\cc(w)$. We apply $f$ to $(\cc(w),\emptyset)$ repeatedly until we get $(\emptyset, \lambda)$, where $\lambda \vdash n$. 
\end{algo}
When applying this algorithm to $\cc(w)$, the resulting shape $\lambda$ is $\ctype(w)$. Furtheremore, we have that if $P(w) = P(w')$, then $\ctype(w) = \ctype(w') = \ctype(P(w))$ \cite{Blasiak}.

As a slight modification to the original algorithm, we record the entries of $\cc(w)$ that correspond to the boxes in $\lambda$ as we build the partition. When reading $i$ results in adding a box to the partition $\nu$, we fill the new box with $i$. This gives us a standard filling (but not a SYT) of shape $\ctype(w)$, meaning $\{1,\dots, n\}$ appear exactly once. We denote this filling as $T_w$. For the modified algorithm, the final tuple is $(\emptyset, T_w)$.

\begin{example}\label{ex: cat alg}
    Consider \begin{align*}
        w &=  3 \ 4  \ 1 \ 2\ 5  \\ 
        \cc(w) &=  1 \ 1 \ 0 \ 0 \ 1 .
    \end{align*}
    We will apply $f$ to $(\cc(w),\emptyset)$ repeatedly until we get an empty word in the first coordinate.
In order to keep track of the indices, we do not rotate $\cc(w)$: instead, we read the word from right to left. The position we are reading at each step of the algorithm is underlined. 
     \ytableausetup{smalltableaux, centertableaux}
     {\allowdisplaybreaks\begin{align*}
          ( 1 \ 1 \ 0 \ 0 \ \mathbf{\underline{1}} \hspace{.1cm}&, \hspace{.3cm} \emptyset) \\ &\downarrow\\ 
     ( 1 \ 1 \ 0 \ \mathbf{\underline{0}} \ 2\hspace{.1cm} &, \hspace{.3cm}\emptyset \hspace{.1cm}) \\ &\downarrow\\ 
     ( 1 \ 1 \  \mathbf{\underline{0}} \ \hspace{.2cm} \ 2 \hspace{.1cm}&, \hspace{.3cm} \ytableaushort{4} \hspace{.1cm}) 
      \\ &\downarrow\\ 
     (1 \ \mathbf{\underline{1}} \  \hspace{.2cm} \ \hspace{.2cm} \ 2 \hspace{.1cm}&,\hspace{.1cm} \ytableaushort{43}\hspace{.1cm}) \\ &\downarrow\\ 
     ( \mathbf{\underline{1}} \  \hspace{.2cm}  \  \hspace{.2cm} \ \hspace{.2cm} \ 2\hspace{.1cm} &, \hspace{.3cm}\ytableaushort{2,43}\hspace{.1cm}) 
     \\ &\downarrow\\ 
     (\hspace{.2cm} \  \hspace{.2cm}  \  \hspace{.2cm} \ \hspace{.2cm} \ \mathbf{\underline{2}}  \hspace{.1cm}&,\hspace{.3cm} \ytableaushort{21,43}\hspace{.1cm}) 
 \\ &\downarrow\\ 
     (  \hspace{.2cm}  \ \hspace{.2cm}  \  \hspace{.2cm} \ \hspace{.2cm} \ \hspace{.2cm} \hspace{.1cm} &, \hspace{.3cm}\ytableaushort{5,21,43}\hspace{.1cm}) 
    \end{align*}}
    From this, we conclude $T_{34125} = \ytableaushort{5,21,43}$ and $\ctype(34125) = (2,2,1)$.
\end{example}

We can use catabolizability type and dominance order on partitions to identify subsets of $\SYT_n$.

\begin{defn}
    For $\lambda\vdash n$, let $\Sigma_\lambda := \{S\in \SYT_n, \ctype(S)\unrhd \lambda\}$.
\end{defn}

Lascoux recognizes $\Sigma_\lambda$ as the image of embedding the cyclage poset of tableau of content $\lambda$ into the cyclage poset of standard tableaux. 

\begin{prop}[Lascoux \cite{Lascoux}]
    For any $\lambda\vdash n$, we have a bijection 
    $\std: \SSYT(\lambda)\rightarrow \Sigma_\lambda\subset \SYT_n$
    which preserves cocharge and shape.
\end{prop}

From this, we can write $\tilde{H}_\lambda[X;q]$ as a sum over $\Sigma_\lambda$\cite{Lascoux, butler}.

\begin{align}\label{eq: frob R_lambda}
    \Frob_q(R_\lambda) = \sum\limits_{S\in \Sigma_\lambda} q^{\cocharge(S)} s_{\shape(S)}.
\end{align}

\begin{example}
    If $\lambda = (2,1)$, we have $\Sigma_\lambda = \left\{\ytableaushort{3,12},\ytableaushort{123}\right\}$. We can see the explicit cocharge/shape bijection between $\Sigma_\lambda$ and $\SSYT(\lambda)  = \left\{\ytableaushort{2,11},\ytableaushort{112}\right\}$.
\end{example}

There is a way to describe this bijection explicitly using crystal operators (for a more detailed reference, see Shimozono-Weyman \cite{SW}). 
Since crystal operators preserve Knuth equivalence, we can extend the bijection $\std$ on tableaux to go from words of weight $\lambda\vdash n$ to permutations $w\in \sym_n$ with $\ctype(w)\unrhd \lambda$. 
From this, we can define the bijection on skew-shaped semistandard tableaux as well. 

\begin{lem}\label{lem: standardization of skew shape}
For any skew shape $\mu/\gamma$ of size $n$ and weight $\lambda \vdash n$, there exists a bijection 
\[\std:\SSYT(\mu/\gamma, \lambda) \to \{S\in \SYT(\mu/\gamma), \ctype(\rw(S)) \unrhd \lambda \}.\]

\end{lem}

Now, for any $w\in \mathfrak{S}_n$ satisfying $\ctype(w) = \lambda$, we can construct $(A_1 | \dots | A_{\lambda_1})\in \osp_{\lambda'}$ by setting 
\begin{align}\label{eq: canon osp}
    A_i := \{j \ | \  j \text{ appears in column } i \text{ of } T_w\}
\end{align}

\begin{prop}[H.{\cite[Proposition 4.4]{hanada2025charge}}]
Let $w\in \mathfrak{S}_n$ with $\ctype(w) = \lambda$ and $(A_1 | \dots | A_{\lambda_1})\in \osp_{\lambda'}$ defined by \eqref{eq: canon osp}. For any choice of $i$, we have $\rev(\cc(w)\vert_{A_i}) \in \mathcal{D}_{\lambda'_i}$.
\end{prop}

\begin{example}
We continue with $w = 34125$ from Example \ref{ex: cat alg}.
    Consider $(245 | 13)\in \osp_{3,2}$, which is the ordered set partition defined by \eqref{eq: canon osp}.  
   Then $\rev(\cc(w)\vert_{245}) = 101\in \mathcal{D}_3$ and $\rev(\cc(w)\vert_{13}) = 01\in \mathcal{D}_2$.
\end{example}

\begin{lem}[H. {\cite[Corollary 4.6]{hanada2025charge}\label{lem: constructing canonical OSP}}]
 Let $w\in \mathfrak{S}_n$ with $\ctype(w) \unrhd \lambda$. There exists $(A_1 | \dots | A_{\lambda_1})\in \osp_{\lambda'}$ such that $\rev(\cc(w)\vert_{A_i}) \in \mathcal{D}_{\lambda'_i}$.
\end{lem}

We recall the construction of such ordered set partition, given in \cite{hanada2025charge}.
Consider a permutation $w$ such that $\ctype(w) = \nu$, where $\nu \unrhd \lambda$ differ by a covering relation in dominance order. Assume we obtain $\lambda$ from $\nu$ by taking the last box in column $j_2$ and moving it to the end of column $j_1$ (where $j_1 < j_2$).

Let $T_w$ be the filling of shape $\nu$ we get from applying Algorithm \ref{alg: ctype} to $w$ and $(A_1|\dots|A_{\nu_1})\in \osp_{\nu'}$ be the ordered set partition we get from \eqref{eq: canon osp}. Let $a$ is the entry in the last box of column $j_2$ in $T_w$. If we set $A'_{j_1} = A_{j_1} \cup \{a\}, A'_{j_2} = A_{j_2} \setminus \{a\}$ and $A'_j = A_j$ for the other $j$, we have that $(A'_1 | \dots | A'_{\lambda_1})\in \osp_{\lambda'}$ satisfies $\rev(\cc(w)\vert_{A_i}) \in \mathcal{D}_{\lambda'_i}$.

Using these properties, along with a counting argument, we get the following result:
\begin{prop}[H. {\cite[Theorem A]{hanada2025charge}}]\label{prop: rev cc and shuffles are the same}
   We have $\mathcal{D}_{\lambda} = \{\rev(\cc(w)) \ | \ w\in \sym_n, P(w)\in \Sigma_\lambda\}$.
\end{prop}

This gives an alternative description of the descent basis of $R_\lambda$. This description will be useful when we construct our bases, since the cardinality of the set is immediate from the presentation.

Finally, we state some lemmas that we will be using in later sections.

\begin{lem}\label{lem: cat-type-containment}
    Let $\lambda \vdash k < n$ be a partition, and $S \in \SYT_n$. Then $\ctype(S \vert_k) \trianglerighteq \lambda$ if and only if $\ctype(S \vert_{k+1}) \trianglerighteq \lambda + (1)$, where $(\lambda_1,\dots,\lambda_l) +(1) = (\lambda_1,\dots, \lambda_l, 1)$.
\end{lem}

\begin{proof}
It suffices to prove the statement on permutations. 
Consider $w\in \mathfrak{S}_n$ such that  $\ctype(w \vert_k) \trianglerighteq \lambda$. Let $U = \{i \ | \ w_i \in [k]\}$.
From Lemma \ref{lem: constructing canonical OSP}, there exists $(A_1|\dots|A_{\lambda_1})\in \osp_{\lambda'}(U)$ such that $\rev(\cc(w\mid_k)\vert_{A_i})\in \mathcal{D}_{\lambda'}$.
Note that this is an ordered set partition of the set $U$, rather than the set $[k]$.

Let $m = w^{-1}_{k+1}$. If we add $m$ to the set $A_j$ that contains $A_k$, we have that the new ordered set partition is in $\osp_{\nu'}(U\cup \{m\})$ for some $\nu\vdash k+1$, where $\nu$ is the result of adding a box to some column of $\lambda$. Since $\nu \unrhd \lambda + (1)$, we have $\ctype(w\mid_{k+1}) \unrhd \lambda + (1)$.

Conversely, assume there exists $(A'_1|\dots|A'_{\lambda_1})\in \osp_{(\lambda + (1))'}(U\cup \{m\})$ such that $\rev(\cc(w)\vert_{k+1}\vert_{A'_i})\in \mathcal{D}_{\lambda'}$.

We know that $\cc(w)\mid_{A'_1} = \cc(\pi)$, for some $\pi\in \mathfrak{S}_{\lambda'_1 + 1}$. From Lemma \ref{lem: insert, delete}(b), we know there exists $r\in A'_1$ such that $\cc(w)\vert_{A'_1\setminus \{r\}} = \cc(\pi\vert_{\lambda'_1})$.

Now, consider $A'_j$ such that $m \in A'_j$. We know that $\cc(w)_r \leq \cc(w)_m$, since $\cc(w)_m$ is the largest entry in $\cc(w)\vert_{k+1}$. Thus by applying Lemma \ref{lem: insert, delete}(a) and then (b), we have that $\cc(w)\mid_{(A'_j \cup \{m\}) \setminus \{r\}} \in \mathcal{D}_{\lambda'_j}$.
Set $A_1 = A'_1 \setminus \{r\}, A_j = (A'_j \cup \{r\}) \setminus \{m\}$, $A_i = A'_i$ otherwise. The resulting ordered set partition shows that $\ctype(w\vert_k) \unrhd \lambda.$
\end{proof}

\begin{lem}\label{lem: ctype-content-shape-ineq}
    Let $S \in \SYT_n$. We have the following inequalities of compositions in dominance order:
    $$ \shape(S) \trianglerighteq \content(\cc(S)) \trianglerighteq \ctype(S)$$
\end{lem}
\begin{proof}
The first inequality follows from the fact that $\cc(S)$ is a semistandard tableau. In particular, we have $\sum\limits_{i=1}^{r} \shape(S)_{i} \geq \sum\limits_{i=1}^{r} \content(\cc(S))_{i-1}$
since all copies of $0,1\dots, (r-1)$ in $\cc(S)$ must appear in rows $r$ or lower in $\cc(S)$.

Similarly, from Algorithm \ref{alg: ctype}, we have
$\sum\limits_{i=1}^{r} \content(\cc(S))_{i-1} \geq \sum\limits_{i=1}^{r} \ctype(S)_{i}$ 
since the entries in $\cc(S)$ that contribute to $\sum\limits_{i=1}^{r} \ctype(S)_{i}$ are at most $(r-1)$.

\end{proof}

\subsection{The $\Delta$-Springer Module $R\nls$ and battery-powered tableaux}
The $\Delta$-Springer modules $R\nls$ were originally defined by Griffin \cite{griffin2021ordered}. In this section, we review some properties of $R\nls$ that we make use of in our work.
The first property is the structure of $R\nls$ as an ungraded $\sym_n$-module. 

\begin{thm}[Griffin {\cite[Theorem 3.21]{griffin2021ordered}}]
    $R\nls \simeq_{\sym_n} \mathbb{Q}\osp\nls$.
\end{thm}

From this, we get an expression for the ungraded Frobenius character of $R\nls$.
\begin{cor}[Griffin {\cite[Theorem 3.21]{griffin2021ordered}}]\label{cor: ungraded Frob}
$\Frob(R\nls) =  \sum\limits_{\substack{\alpha = (\alpha_1,\dots,\alpha_s)\models n, \\ \lambda\subset \alpha}} h_\alpha$.
\end{cor}
Now, we can rewrite Corollary \ref{cor: ungraded Frob} by expanding $h_\alpha$ into Schur functions using the set  $\SSYT(\alpha)$, which denotes the collection of all semistandard tableaux of weight $\alpha$:
\begin{align}\label{eq: Frob Rnls wrt ssyt}
    \Frob(R\nls) = \sum\limits_{\substack{\alpha = (\alpha_1,\dots,\alpha_s)\models n, \\ \lambda\subset \alpha}}\sum_{T\in \SSYT(\alpha)} s_{\shape(T)}.
\end{align}

Using  \eqref{eq: Frob Rnls wrt ssyt} , we can easily get an expression for $\dim(R\nls)$ in terms of semistandard Young tableaux, by replacing $s_{\shape(T)}$ with $K_{\shape(T),1^n}$.
\begin{cor}\label{cor:dim} $\dim(R\nls) = \sum\limits_{\substack{\alpha = (\alpha_1,\dots,\alpha_s)\models n, \\ \lambda\subset \alpha}} \sum\limits_{T\in \SSYT(\alpha)} K_{\shape(T), 1^n}$

\end{cor}
We can reformulate Corollary \ref{cor:dim} using a combinatorial object called \emph{battery-powered tableaux}, originally defined by Gillespie--Griffin in \cite{battery}. 
\begin{defn}
 A \textit{battery-powered tableau} of parameters $(n,\lambda,s) $ consists of a pair $T = (D, B)$ of semistandard Young tableaux, where $B$ is of shape $(n-k)\times (s-1)$ and the total content of $(D,B)$ is given by $\Lambda\nls:= (n-k)\times s + \lambda$.
 \end{defn}
 We refer to $D,B$ as the \textit{device} and \textit{battery} of $T$. Note that a battery-powered tableaux is uniquely defined by the device: that is, once we know $D$, there is at most one $B$ so that the total tableau has the correct weight. Let $\mathcal{T}^+\nls$ denote the set of all battery-powered tableaux of parameters $(n,\lambda,s)$. 

\begin{example}\label{ex: bat tab}
Let $(n,\lambda, s) =(4 ,\ydiagram{1,2},3)$. 
Then $T = \ytableaushort{2,112,\none\none\none 3, \none\none\none 1}\in \mathcal{T}^+\nls$.  
\end{example}

\begin{lem}\label{lem: bij, battery and ssyt (no cc)}
    The map $\mathcal{T}^+\nls  \to \bigcup\limits_{\substack{{\alpha = (\alpha_1,\dots,\alpha_s)\models n} \\ {\lambda\subset \alpha}}} \SSYT(\alpha)$
 that takes $(D,B)\mapsto D$ is a bijection. 
\end{lem}
\begin{proof}
     Note that for any battery-powered tableau $(D,B)$, it is clear that $\weight(D) = \alpha$ for composition $(\alpha_1.\dots,\alpha_s)\models n$ such that $\lambda\subset \alpha$.
     We can also see that for any $D\in \SSYT(\alpha)$, there is exactly one way to fill in the rectangle shape $B$ so that the skew-shape $(D,B)$ has total weight $\Lambda\nls$.
\end{proof}

By applying the bijection in Lemma \ref{lem: bij, battery and ssyt (no cc)} to Corollary \ref{cor:dim}, we get a new expression for $\dim(R\nls)$ in terms of these battery-powered tableaux.

\begin{cor}\label{cor: dim of Rnls}
    $ \dim(R\nls)  = \sum\limits_{(D,B)\in \mathcal{T}^+\nls} K_{\shape(D), 1^n}.$
\end{cor}
We will use Corollary \ref{cor: dim of Rnls} when showing that our set of monomials has the right count to be a basis of $R\nls$.

The battery-powered tableaux were originally introduced in \cite{battery} to give a combinatorial formula for $\Frob_q(R\nls)$:
\begin{thm}[Gillespie--Griffin {\cite[Theorem 1.6]{battery}}]\label{thm: Frob, battery}
    \begin{align}
        \Frob_q(R\nls) = \frac{1}{q^{\binom{s-1}{2}(n-k)
}} \sum\limits_{T = (D,B) \in \mathcal{T}^+\nls} q^{\cocharge(T)} s_{\shape(D)}.
\end{align}
\end{thm}

\begin{example}
    For $(n,\lambda,s) = (4,\ydiagram{1,2}, 3)$, the set of battery-powered tableaux are:
    \begin{align*}
\ytableaushort{1123,\none\none\none\none2,\none\none\none\none1}, \hspace{.5cm} \ytableaushort{1122,\none\none\none\none3,\none\none\none\none1}, \hspace{.5cm} \ytableaushort{1112,\none\none\none\none3,\none\none\none\none2}, \\ 
\ytableaushort{3,112,\none\none\none2,\none\none\none1}, \hspace{.5cm} \ytableaushort{2,112,\none\none\none3,\none\none\none1}, \hspace{.5cm} \ytableaushort{2,111,\none\none\none3,\none\none\none2},\\
\ytableaushort{23,11,\none\none2,\none\none1}, \hspace{.5cm} \ytableaushort{22,11,\none\none3,\none\none1}\hspace{.5cm} \ytableaushort{2,113,\none\none\none2,\none\none\none1}, \hspace{.5cm} \ytableaushort{3,2,11,\none\none2,\none\none1}.
\end{align*}
From this, we know that $\Frob_q(R\nls) = (1+q+q^2) s_{4} + (q+2q^2+q^3) s_{3,1} + (q^2+q^3)s_{2,2}+ q^3 s_{2,1,1}$.
\end{example}

\subsection{Higher Specht Bases}\label{subsec: higher specht background}
We recall the higher Specht basis defined by Ariki-Terasoma-Yamada \cite{ariki1997higher} for the coinvariant algebra $R_n$.

Let $S,T \in \SYT_n$ and let $\sh(S) = \sh(T) = \lambda$. We consider the monomial $\xx_T^{\cc(S)}$ obtained by superimposing $T$ with the cocharge labeling of $S$:

$$ \xx_T^{\cc(S)} := \prod_{u \in \lambda} x^{\cc(S)(u)}_{T(u)}$$
where $\cc(S)$ denotes the cocharge labeling of $S$ (Definition \ref{def: cc(S)}) and $T(u)$ is the entry in the box $u \in \lambda$ for the filling $T$.

If $T \in SYT(\lambda)$, denote by $C(T)$ (resp: $R(T)$) to be the group of column (row) permutations, or $\tau \in \sym_n$ (resp: $\sigma \in \sym_n$) which preserves the columns (rows) of $T$. Define the \emph{Young idempotent} to be the group algebra element

$$\varepsilon_T := \sum_{(\sigma,\tau) \in R(T) \times C(T)} (-1)^\tau \tau\sigma$$

Given $S,T \in \SYT(\lambda)$, the \emph{higher Specht polynomial} for $(S,T)$ is defined to be

$$ F_T^S := \varepsilon \cdot \xx_T^{\cc(S)}$$

\begin{thm}[{Ariki--Terasoma--Yamada\cite{ariki1997higher}}]
     The collection of polynomials $$\CC_n := \{ F_T^S: S,T \in \SYT_n, \sh(S) = \sh(T) \}$$ descend to a higher Specht basis of $R_n$.
\end{thm}

We have a generalization of this basis to $R_{n,k}$, due to Gillespie-Rhoades \cite{RhoadesGillespie2021}.

\begin{thm}[{Gillespie--Rhoades\cite{RhoadesGillespie2021}}]\label{thm: higher specht Rnk}
    The following collection of polynomials \[\mathcal{C}_{n,k} = \{F_T^S\cdot e_1^{i_1}e_2^{i_2}\cdots e_{n-k}^{i_{n-k}}:  S,T\in \SYT_n, \sh(S) = \sh(T), i_j\in \mathbb{Z}_{\geq 0}, 0\leq i_1+\cdots i_{n-k}<k-\des(S)\}\]
  descends to a higher Specht basis of $R_{n,k}$.
\end{thm}

\section{Tableaux Construction for $R_{n,k}$}\label{sec: tab constuction for R_{n,k}}
In this section, we reformulate the generalized descent basis of $R_{n,k}$ given in Theorem \ref{thm: HRS descent} using a generalization of cocharge to match Proposition \ref{prop: descent is rev cc} (for the coinvariant ring) and Proposition \ref{prop: rev cc and shuffles are the same} (for $R_\lambda$). 

First, we define the following indexing set:
\begin{align}\label{def: sigma n,k}
    \Sigma_{n,k} = \{(S,\mu): S \in \SYT_n, \mu \subseteq  (n-k) \times (k-\des(S)-1)\}.
\end{align}

Note that we can take a tuple of nonnegative integers $(i_1,\dots, i_{n-k})$ such that $i_1+\cdots+i_{n-k}<k-\des(S)$ and write it as a partition $\mu \subseteq (n-k)\times (k-\des(S)-1)$  by setting $\mu_{j} = i_j + \cdots + i_{n-k}$. Thus,  this is a reformulation of the sets that index the descent basis in Theorem \ref{thm: HRS descent} and the higher Specht basis in Theorem \ref{thm: higher specht Rnk}.

Now, we define a map on the set $\{(w,\mu) \ | \ w\in S_n, (P(w),\mu)\in \Sigma_{n,k}\}$ for any $k\leq n$ which gives us a \emph{boosted cocharge word}: that is, it will be the result of increasing some values in $\cc(w)$ according to the partition $\mu$.

Given $(w,\mu)$, let $\cc(w,\mu)$ be the word of length $n$ defined by
\begin{align}\label{defn: Phi map}
    \cc(w,\mu)_i &=\begin{cases}
     \cc(w)_i + \mu_{n-w_i + 1}
    & \text{ if } w_i \in \{k+1,\dots, n\} \\ \cc(w)_i  & \text{otherwise}
    \end{cases}
\end{align}

In other words, we can think of $\cc(w,\mu)_i$ as the word we get from ``boosting'' the cocharge of entries $k+1,\dots, n$ in $w$ according to $\mu$.

\begin{example}
    Let $w = 7136524$ and $\mu = (3,2,2)$. Then $\cc(w) = 4013201$ and $\cc(w,\ix) = 7014301$.
\end{example}

We can also define this notion of boosted cocharge directly on $\Sigma_{n,k}$. 
 Given $(S,\mu)\in \Sigma_{n,k}$, let $\cc(w,\mu)$ be a semistandard tableaux of the same shape obtained by taking $\cc(S)$ and adding $\mu_{n-j+1}$ to $\cc(S)_{j}$ for $j\in \{k+1,\dots,n\}$.
We have $\cc(\rw(S),\mu) = \rw(\cc(S,\mu))$ for any $(S,\mu)\in \Sigma_{n,k}$. 
We can express the set $\mathcal{GS}_{n,k}$ in Theorem \ref{thm: HRS descent} using these boosted cocharge words $\cc(w,\ix)$.

\begin{prop}\label{prop: des-same-basis-hrs}
    We have that \[\mathcal{GS}_{n,k} = \{\xx^{\rev(\cc(w,\mu))} : w\in \sym_n, (P(w),\mu)\in \Sigma_{n,k}\}.\]
\end{prop}

\begin{proof}
Consider $(w,\mu)$ where $(P(w),\mu)\in \Sigma_{n,k}$. Note that we set $i_{n-k} = \mu_{n-k}$, $i_{j}= \mu_j - \mu_{j+1}$ for $j< n-k.$
We can see that for $\sigma = \rev((\rev(w))^{-1})$,we have
\begin{align*}
\prod\limits_{j=1}^{n-k} (x_{\sigma_1}\dots x_{\sigma_{j}})^{i_j} &= 
\prod\limits_{j=1}^{n-k} x_{\sigma_j}^{i_j + \cdots+i_{n-k}} \\
 &=  \prod\limits_{j=1}^{n-k} x_{n-w^{-1}_{n-j+1} +1}^{i_j + \cdots+i_{n-k}},
 \\
  &=  \prod\limits_{j=1}^{n-k} x_{n-w^{-1}_{n-j+1} +1}^{\mu_j}.
\end{align*}
Thus $g_\sigma(\xx)\prod\limits_{j=1}^{n-k} (x_{\sigma_1}\dots x_{\sigma_{j}})^{i_j}  = \xx^{\rev(\cc(w,\mu))}$. 
Furthermore, by definition of $\Sigma_{n,k}$, we know that $\des(P(w)) < k$. From Lemma \ref{lem: max cc}, this is equivalent to saying the maximal cocharge value in $\cc(w)$ is at most $k-1$. In other words, we have $\des(w^{-1})< k$.
This condition is equivalent to the condition $\des(\sigma) < k$ on $\sigma$, since 
\[\des(\sigma)< k \Longleftrightarrow \des(\pi^{-1}) \geq  n-k+1 \Longleftrightarrow  \des(w^{-1}) < k.\]
where $\pi = \rev(\sigma)^{-1}$.
\end{proof}

\section{Combinatorial Constructions}\label{sec:counting}

Now, we show that our combinatorial indexing set $\DD\nls$ has the correct cardinality, namely $|\DD\nls| = \dim R\nls$. We restate the definition of the indexing set here letting $[i]_0 := \{0,1,\dots,i\}$ throughout:

\begin{defn}\label{def: dnls}
For any $\lambda\vdash k\leq n$, define $\osp_{n,\lambda'}$ to be the collection of ordered set partitions $\sigma  = (A_1| \dots| \ A_{\lambda_1} \ | \  B_1 \ |\dots|\  B_{n-k})$ of $[n]$ where $|A_i| = \lambda'_i$ and $|B_j| = 1$ for any $i,j$. We denote $\sigma = (\AAAA_\lambda | \BBB)$, and denote $\BBB = \bigsqcup_{i = 1}^{n-k} B_i$. We define the \emph{$n,\lambda,s$-generalized descent words} to be
\begin{align}\label{eq: Dnls}
    \mathcal{D}\nls: = \{\ax : \ax\vert_{A_i} \in \mathcal{D}_{\lambda'_i}, \ax\vert_{B_j} \in [s-1]_0  \text{ for some } \sigma \in \osp_{n,\lambda'}\}.
\end{align}
\end{defn}

\begin{example}
    Consider $(n,\lambda,s) = (4,(2,1),3)$.
    Then 
    \begin{align*}
        \mathcal{D}\nls = \{0000,&1000,0100,0010,0001, 2000,0200,0020,0002,\\&0101,0110,0011,1001,1010,0101,0120,0012,2001,2010,0201,0210\}.
    \end{align*}
  where $\mathcal{D}_1 = \{0\}, \mathcal{D}_2 = \{00,01\}.$
  
\end{example}
Though the construction is straightforward, it is difficult to directly count the cardinality of $\mathcal{D}\nls$. In particular, we can have multiple ordered set partitions that satisfy the condition in \eqref{eq: Dnls} for the same word.  For $\mathbf{z} = 0110$ in the example above, we have $(12|4|3)$ and $(13|4|2)$ both work. 

To get around this issue and show that $|\mathcal{D}\nls| = \dim(R\nls)$, we introduce a new indexing set $\Sigma\nls$.
\begin{align}\label{def: sigma nls}
    \Sigma\nls := \{(S,\mu) :S\in \SYT_n, S\vert_k \ \in \Sigma_\lambda, \mu\subseteq (n-k)\times(s-\des(S)-1)\}.
\end{align}
Note that we do not consider $S$ where $s-\des(S)<0$. Our indexing set specializes naturally to the indexing sets of $R_\lambda$, given in Proposition \ref{prop: rev cc and shuffles are the same}, and $R_{n,k}$, given in Proposition \ref{prop: des-same-basis-hrs}: it is immediate that $\Sigma_{k,\lambda,\ell(\lambda)} = \Sigma_{\lambda}$, and $\Sigma_{n, 1^k,k} = \Sigma_{n,k}$. In fact, the introduction of the set $\Sigma\nls$ to directly see the cardinality of $\mathcal{D}\nls$ is analogous to how \cite{hanada2025charge} introduced a way to directly count the monomials constructed in \cite{carlsson2024descent}.

\begin{example}
    Consider $(n,\lambda,s) = (4,(2,1),3)$. The possible pairings $(S,\mu)$ are recorded in the table below:
    \begin{center}
         \begin{tabular}{c|c|c}
        $S$ & $s-\des(S)$  & \text{possible} $\mu\subseteq 1\times (2-\des(S))$ \\ \hline 
        $\ytableaushort{1234}$ &  3 & $\emptyset, \ydiagram{1} ,  \ydiagram{2}$ \\  $\ytableaushort{4,123}$ &  $2$ & $\emptyset, \ydiagram{1} $\\ $\ytableaushort{34,12}$ &  $2$ & $\emptyset, \ydiagram{1}$ \\ $\ytableaushort{3,124}$ &  $2$ & $\emptyset, \ydiagram{1}$ \\ $\ytableaushort{4,3,12}$ &  $1$ & $\emptyset$.
    \end{tabular}
    \end{center}
\end{example}

Now, we construct a bijection between pairs $(w,\mu)$ where $(P(w),\mu)\in \Sigma\nls$ and our indexing set $\mathcal{D}\nls$, using the notion of boosted cocharge words.

\begin{prop}\label{lem: Dnls to pairs}
    There exists a bijection $\{(w,\mu) \ | \ w\in \mathfrak{S}_n, (P(w),\mu)\in \Sigma\nls\} \to \mathcal{D}\nls.$
\end{prop}
\begin{proof}
   Consider the map that takes $(w,\mu)\mapsto\rev(\cc(w,\mu))$. We claim that this gives such a bijection.
    
    We first check that $\rev(\cc(w,\mu))) \in \mathcal{D}\nls$. Note that it suffices to construct $(\mathbf{A}_\lambda|\mathbf{B})\in \osp_{n,\lambda'}$ such that $\rev(\cc(w,\mu)\vert_{A_i}) \in \mathcal{D}_{\lambda'_i}$ and $\rev(\cc(w,\mu)\vert_{B_j}) \in [s-1]_0$.

    First, note that  $\cc(w,\mu)_i = \cc(w)_i$ if $w_i \in [k]$.
    Let $ U = \{i \ | \ w_i \in [k]\}$. We can see that $\cc(w,\mu)\vert_{U} = \cc(w\vert_k)$ .
    By Proposition \ref{prop: rev cc and shuffles are the same}, we know $\rev(\cc(w\vert_k)) \in \mathcal{D}_{\lambda}$.
    Thus, by the definition of $\mathcal{D}_\lambda$, there exists some $\mathbf{A}_\lambda\in \osp_{\lambda'}(U)$, where $\mathbf{A}_\lambda$ is an ordered set partition of the set $U$ such that  $\rev(\cc(w,\mu)\vert_{A_i}) \in \mathcal{D}_{\lambda'_i}$.
    
    Furthermore, for $i$ such that $w_i\in [k+1,n]$ by Lemma \ref{lem: max cc} we have the following relation: \begin{align}\label{eq: bound on entries}
        0\leq \cc(w,\mu)_i = \cc(w)_i + \mu_{n-w_i+1} < \cc(w)_i + s-\des(S) \leq  s.
    \end{align}
    Thus $\cc(w,\mu)_i \leq s-1$.
    Let $B_j = \{i \ | \ w_i = k+j\}$. Then $\rev(\cc(w,\mu)\vert_{B_j}) \in  [s-1]_0$ by \eqref{eq: bound on entries}.
    Thus $\rev(\cc(w,\mu)) \in \mathcal{D}\nls$.

    Now we construct an inverse map. Let $\rev(\mathbf{u})\in\mathcal{D}\nls$. 
    We know there exists $(\mathbf{A}_\lambda|\mathbf{B})\in \osp_{n,\lambda'}$ such that $\rev(\mathbf{u}\vert_{A_i}) \in \mathcal{D}_{\lambda'_i}, \mathbf{u}\vert_{B_j} \in [s-1]_0$. 

    Note that we can put a total order (the \textit{reading order}) on the set $[n]$ in the following way:
    \begin{align}\label{eq: reading order}
        i <_{\mathbf{u}} j &\Leftrightarrow \begin{cases}
            u_i < u_j \\ u_i=u_j, i<j.
        \end{cases}
    \end{align}

    Let $U\subset [n]$ be the subset consisting of the smallest $k$ entries with respect to $<_\mathbf{u}$.
    We want to construct a new $(\mathbf{A}'_\lambda|\mathbf{B'})\in \osp_{n,\lambda'}$ such that $\mathbf{A'}_\lambda$ is an ordered set partition of $U$ and $\rev(\mathbf{u}\vert_{A'_i}) \in \mathcal{D}_{\lambda'_i}, \mathbf{u}\vert_{B'_j} \in [s-1]_0$ still holds.

    We do this by modifying the original ordered set partition $(\mathbf{A}_\lambda |\mathbf{B})$. 
    Let $m$ be the largest entry with respect to $<_\mathbf{u}$. 
    If $m\not\in A_i$ for all $i$, we do not modify the ordered set partition. Otherwise, if $m\in A_i$ for some $i$, there exists a $c$ such that $c\in U$ but $B_j = \{c\}$ for some $j$.
    
    Since $c<_\mathbf{u}m$, we know that $u_c \leq u_m$ which implies $\rev(\mathbf{u}\vert_{A_i\cup\{c\}}) \in \mathcal{D}_{\lambda'_i+1}$ from  Lemma \ref{lem: insert, delete}(1).
    Furthermore, we know that $u_m$ is the rightmost,  largest entry in $\mathbf{u}\vert_{A_i\cup\{c\}}$.
    Thus, by Lemma \ref{lem: insert, delete} (2), we have that $\rev(\mathbf{u}\vert_{(A_i \cup\{c\})\setminus \{m\}})\in \mathcal{D}_{\lambda'_i}$.

    Hence, if we let $(\mathbf{A}'_\lambda|\mathbf{B'})$ be the ordered set partition we get by setting 
    $A'_i = (A_i \cup\{c\})\setminus \{m\}, B'_j = \{m\}$ without changing the rest of $(\mathbf{A}_\lambda|\mathbf{B})$, we still have  $\rev(\mathbf{u}\vert_{A'_i})\in \mathcal{D}_{\lambda'_i}$. 
    Furthermore, since $u_m \leq \lambda'_i < s$, we still have $\mathbf{u}\vert_{B_j}\in [s-1]_0$ for all choices of $j$.

    We can repeat this process for each entry not in $S$, until we get $(\mathbf{A}'_\lambda|\mathbf{B'})$ where $\mathbf{A}'_\lambda$ is an ordered set partition of $S$.

    Now, we construct $(w,\mu)$ such that $\cc (w,\mu) = \mathbf{u}$. If $i$ is the $d$th smallest entry with respect to $<_\mathbf{u}$, we set $w_{i} = d$.
    By construction, we have that $\cc(w\vert_k) = \rev(\mathbf{u}\vert_{\mathbf{A}'_\lambda})$, thus $P(w\vert_k) \in \Sigma_{\lambda}$ by Proposition \ref{lem: RSK restriction}.
    
    Set $\mu_{n-k+1} = 0$. For any $j\leq n-k$, let $a = w^{-1}_{n-j}$,  $b = w^{-1}_{n-j+1}$. Note that $a,b$ are the $(n-j)$th, $(n-j+1)$th smallest letters with respect to $<_{\mathbf{u}}$. We can define $\mu_{j}$ to be
  \begin{align}
      \mu_j:= \begin{cases}
        \mu_{j+1} + u_b- u_a & \text{ if }  a< b\\
        \mu_{j+1} + u_b- u_a -1& \text{ if }  a> b.
        \end{cases}
  \end{align}

    The first (resp. second) case is when $(n-j)$ appears to the left (resp.right) of $(n-j+1)$ in $w$.  Note that in either case, we have $\mu_j \geq \mu_{j+1} \geq 0$. We can also see that $\cc(w,\mu) = \mathbf{u}$ and $\mathbf{u}_m = \cc(w)_{m} + \mu_1 < s$ implies $\mu_1< s-\des(S)$, thus $\mu \subseteq (n-k)\times (s-\des(S)-1)$.
\end{proof}

Thus we have a bijection between our boosted cocharge words coming from $\Sigma\nls$ and $\mathcal{D}\nls$.
Now, we want to relate the set $\Sigma\nls$ to the set of battery-powered tableaux $\mathcal{T}^+\nls$. First, we define the set of \emph{standardized} battery-powered tableaux, which is just the image of the battery-powered tableaux under the standardization map that preserves cocharge and shape.

\begin{defn}
    Let $\std(\mathcal{T}^+\nls)$ be the image of the set $\mathcal{T}^+\nls$ under Lascoux standardization. 
\end{defn}

For $T\in \mathcal{T}^+\nls$,  let $\std(T)\vert_D, \std(T)\vert_B$ denote the device and battery part of $\std(T)$. 
Let $\cc(\std(T))\vert_D, \cc(\std(T))\vert_B$ denote the device and battery part of $\cc(\std(T))$. 

\begin{example}\label{ex: std}
       Continuing with the battery powered tableau $T$ from Example \ref{ex: bat tab}, we have
       $\std(T) = \ytableaushort{4,125,\none\none\none 6, \none\none\none 3}$, where $\std(T)\vert_D= \ytableaushort{4,125}$, $\std(T)\vert_B = \ytableaushort{6,3}$.
\end{example}

\begin{lem}\label{lem: cc of battery}
 The tableau $\cc(\std(T))\vert_{B}$ is the superstandard tableau of shape $(n-k)\times (s-1)$: that is, the semistandard tableau with row $i$ filled with $(i-1)$.
\end{lem}

\begin{proof}
    Assume otherwise. Note that since we know $\ctype(\std(T))\unrhd \weight(T) = (n-k)\times s + \lambda$, by Algorithm \ref{alg: ctype}, we know that the largest entry in $\cc(\std(T))$ is at most $s-1$.
    Since $\cc(\std(T))$ is a semistandard Young tableau and the battery has $s-1$ many rows, this implies that the top right corner of the battery contains $s-1$.

    Furthermore, row $i$ of the battery contains only $(i-1)$'s and $i$'s. When we apply Algorithm \ref{alg: ctype}, we see that reading the $i$'s in row $i$ does not result in adding a box to the partition $\nu$. Thus we add 1 to $i$ and move it to the end of the word. 
    From this, we know that the $(s-1)$ in the top right corner of the battery ends up contributing a box to row $s+1$ or higher in the final partition $\ctype(\std(T))$. However, this also contradicts $\ctype(\std(T))\unrhd \weight(T) = (n-k)\times s + \lambda$.
\end{proof}

\begin{example}
    Continuing with Example \ref{ex: std}, we can see that
       $\cc(\std(T)) = \ytableaushort{1,001,\none\none\none 1, \none\none\none 0}$.
\end{example}

Now, we construct a map $\Psi$ on $\Sigma\nls$.  Consider $(S,\mu)\in \Sigma\nls$.
Recall that for any $(S,\mu)$ we can construct a boosted cocharge tableau $\cc(S,\mu)$. 
Let $\tilde{\Psi}(S,\mu)$ be the tableau we get by appending the superstandard tableau of shape $(n-k)\times (s-1)$ to the bottom right corner of $\cc(S,\mu)$.

\begin{example}
    Consider $\left(\ytableaushort{34,12},\ydiagram{1}\right)\in \Sigma_{4,(2,1),3}$.Then we have
    \[\cc(S,\mu) = \ytableaushort{12,00} \hspace{2cm} \tilde{\Psi}(S,\mathbf{i}) = \ytableaushort{12,00,\none\none 1,\none\none 0}.\]
\end{example}

\begin{lem}
    For $(S,\mu)\in \Sigma\nls$, we have that $\tilde{\Psi}(S,\mu)$ is equal to $\cc(\std(T))$ for some $T\in \mathcal{T}^+\nls$.
\end{lem}
\begin{proof}
Note that $\tilde{\Psi}(S,\mu)$ contains at least one of $0,\dots, s-2$, since they are contained in the superstandard tableau we appended to the bottom right corner. We can see that for any $c\in [s-3]_0$, there exists a $(c+1)$ in some row higher than the lowest $c$, since we know the lowest $c$ is in row $(c+1)$ of the battery and there is a $(c+1)$ in the row above it. Furthermore, we know the entries in the device part of $\tilde{\Psi}(S,\mu)$ are at most $s-1$, since the largest entry is $\des(S) + \mu_1 \leq s-1$. Thus, either $(s-2)$ is the largest entry in $\tilde{\Psi}(S,\mu)$ , or there exists a $(s-1)$ in a row higher than the lowest $(s-2)$.
Thus $\tilde{\Psi}(S,\mu)$ is the cocharge tableau for some standard tableau by Lemma \ref{lem: cc(S) condition}.

Furthermore, set $\rw(\tilde{\Psi}(S,\mu))= \mathbf{d}\mathbf{b}$, where $\mathbf{d}$ is the reading word of the top semistandard tableau and $\mathbf{b}$ is the reading word of the superstandard tableau appended at the bottom right. We can see that $\rev(\dx) \in \mathcal{D}\nls$ by Proposition \ref{lem: Dnls to pairs} since $\rev(\dx) = \cc(w,\mu)$. In particular, there exists a subset $U\subset [n]$ such that $\rev(\dx\vert_U) \in \mathcal{D}_\lambda$ where $[n] \ \setminus \  U = \{a_1,\dots, a_{n-k}\}$.
Furthermore, we know that $\mathbf{b}$ can be broken up into $n-k$ many disjoint subwords of the form $(s-2)(s-1)\dots 210$.
Then by Lemma \ref{lem: insert, delete}, we know that $\rev(d_{a_i}(s-2)(s-1)\dots 210)\in \mathcal{D}_s$. From this, we know that we can break up $\rev(\dx\bx)$ into $\rev(\dx\vert_U)\in \mathcal{D}_\lambda$ and $(n-k)$ many words, each in $\mathcal{D}_s$. 
Thus $\rev(\dx\bx)\in \mathcal{D}_{(n-k)^s + \lambda}$.
 Thus by Proposition \ref{prop: rev cc and shuffles are the same} and Lemma \ref{lem: standardization of skew shape}, we have $\tilde{\Psi}(S,\mu) = \std(T)$ for some $T\in \mathcal{T}^+\nls$.
\end{proof}

Now, let $\Psi(S,\mu) = T \in \mathcal{T}^+\nls$ where  $\tilde{\Psi}(S,\mu) = \std(T)$.

\begin{prop}\label{prop: bij psi}
    We have a bijection
    \[\Psi: \Sigma\nls \to \std(\mathcal{T}^+\nls) \to \mathcal{T}^+\nls\]
    such that the device of $\Psi(S,\mu)$ is the same shape as $S$ and $\cocharge(\Psi(S,\mu)) = \cocharge(S) + |\mu|$.
\end{prop}

We prove this by constructing an inverse map $\Psi^{-1}$. Note that since $\Psi$ is a tableaux analogue of the map $(w,\mu)\mapsto \rev(\cc(w,\mu))$, the inverse map $\Psi^{-1}$ will involve the inverse map constructed in Proposition \ref{lem: Dnls to pairs}. For any $\dx\in \mathcal{D}\nls$, let $\cc^{-1}(\rev(\dx)) = (w,\mu)$ where $\rev(\cc(w,\mu)) = \dx$. 
\begin{prop}\label{prop: device is right}
    Let $\cc(\rw(\std(T))) = \dx \bx$, where $\dx, \bx$ are the subwords corresponding to the entries in the device and battery respectively. 
    Then $\rev(\dx)\in \mathcal{D}\nls$.
\end{prop}

\begin{proof}
  We know $\ctype(\std(T))\unrhd (n-k)^s + \lambda$, which implies that $\rev(\dx\bx) = \rev(\bx)\rev(\dx) \in \mathcal{D}_{(n-k)^s + \lambda}$. From this, we know that there exists $(C_1 \ | \ \dots \ | \  C_{n-k} \ | \ \mathbf{A}_{\lambda})\in \osp_{k+(n-k)s, ((n-k)^s + \lambda)'}$ where each restriction satisfies the condition in \eqref{eq: Dnls}. 

  By Lemma \ref{lem: constructing canonical OSP}, we know that $\rev(\dx\bx\vert_{C_i}) = 0 1 \dots (s-1) m_i$, where $m_i\in [s-1]_0$ and the subword $01\dots(s-1)$ consists of letters in $\rev(\bx)$. 

  If we set $B_i = \{\text{index corresponding to } m_i\}$, we can construct a new ordered set partition $([(n-k)(s-1)] | \mathbf{A}_\lambda | \mathbf{B})$. Since $(\rev(\dx\bx)) \vert_{[(n-k)(s-1)]} = \rev(\bx)$, we see that $(\mathbf{A}_\lambda | \mathbf{B}))$ gives an ordered set partition of the indices of $\rev(\dx)$ such that \eqref{eq: Dnls} holds.
\end{proof}

\begin{lem}\label{lem: Phi inverse is rw too}
$\cc^{-1}(\rev(\dx)) = (\rw(S), \mu)$ for some $S\in \Sigma\nls$.
\end{lem}
\begin{proof}
    This follows from the fact that we can see $\cc^{-1}(\rev(\dx))$ directly on the tableau $\cc(\std(T))\vert_D$.
    Note that since $\std(T)$ is a standard tableau, there is a total ordering on the boxes in $\cc(\std(T))|D$, given by the ordering of the entries in $\std(T)$. 

    This ordering matches the reading order of $\dx$ specified in \eqref{eq: reading order}. 
\end{proof}

\begin{proof}[Proof of Proposition \ref{prop: bij psi}]
It suffices to construct an inverse map. Take $\Psi^{-1}(T) = (S,\mu)$, where $\cc(\std(T)) = \dx \bx$ and $\cc^{-1}(\rev(\dx)) = (\rw(S),\mu)$ from Lemma \ref{lem: Phi inverse is rw too}. We can check that the composition is the identity using the fact that $(w,\mu)\mapsto \rev(\cc(w,\mu))$ itself is a bijection.

The second statement follows from the fact that $\cocharge(\Psi(S,\mu))$ is the sum of the entries in $\tilde{\Psi}(S,\mu)$. From Lemma \ref{lem: cc of battery}, we know the sum of the entries in the battery component is $\binom{s-1}{2}(n-k)$. The sum of the entries in the device component is $\cocharge(S) + |\mu|$.
\end{proof}

\begin{cor}\label{cor: count}
    We have that $\dim (R\nls) = \mathcal{D}\nls$.
\end{cor}

\begin{proof}
From Proposition \ref{prop: bij psi} and using Corollary \ref{cor: dim of Rnls} and Proposition \ref{lem: Dnls to pairs}, we see that  $$|\mathcal{D}\nls| = \sum\limits_{(S,\mu)\in \Sigma\nls} K_{\shape(S),1^n}  = \dim(R\nls).$$ 
\end{proof}

\begin{rem}
    Note that though the map $(D,B)\mapsto D$ defined in Lemma \ref{lem: bij, battery and ssyt (no cc)} is the most straightforward bijection between $\mathcal{T}^+\nls$ and the collection of semistandard tableaux, it does not behave well with respect to cocharge. In particular, there is no clear relation between the cocharge of the two tableaux like the one in Proposition \ref{prop: bij psi}, which makes it difficult to relate the formula \eqref{eq: Frob Rnls wrt ssyt} to one using battery-powered tableaux. In Section \ref{sec: Frob char}, we highlight how using the set $\Sigma\nls$, which involves standard tableaux, resolves this issue. 
\end{rem}

  Note that using Proposition \ref{lem: Dnls to pairs}, we can restate Theorem \ref{thm: A}:
  \begin{thma}[Cocharge version]
       The set $\BB\nls = \{\xx^{\rev(\cc(w,\mu))}: w\in \sym_n, (P(w),\mu)\in \Sigma\nls\}$ descends to a monomial basis of $R\nls$.
  \end{thma}

  Using this formulation and the fact that $\Sigma_{k,\lambda,\ell(\lambda)} = \Sigma_{\lambda}$,$\Sigma_{n,1^k,k} = \Sigma_{n,k}$ , it is evident that our descent basis construction specializes to the known descent bases for $R_\lambda, R_{n,k}$.

\section{Descent Basis}\label{sec:des}
This section will be dedicated to giving a proof of Theorem \ref{thm: A}, by showing the collection of monomials is linearly independent. We follow the argument in \cite{carlsson2024descent}, of which we give a brief outline here. 

Denote by $\widehat{\lambda}' = (\lambda_2',\dots,\lambda_m')$ the partition obtained by removing the first part of $\lambda'$, which corresponds to taking off the first column of $\lambda$. Given any subset $A \subset [n]$ of size $|A| = \lambda_1'$, we will show that there is a well-defined map
$$ \varphi_{A}: R\nls \to R_{\lambda'_1} \otimes R_{n,\widehat{\lambda},s}.$$
Then, we these maps to inductively show that if there is a relation of the form
$$ \sum_{\bx \leq_{\des} \ax} c_\bx \xx^\bx = 0$$
with $\ax \in \DD\nls$, then we must have that $c_\ax = 0$. This implies the monomials $\xx^\ax \in \BB\nls$ are linearly independent in $R\nls$. Applying Corollary \ref{cor: count} upgrades $\BB\nls$ to a basis.

\begin{prop}\label{prop: phi-delete-lambda1}
    Let $\lambda \vdash k \leq n$, $s \geq \ell(\lambda)$, and $\widehat{\lambda}$ be as above. Given any subset $A \subset [n]$ of size $|A| = \lambda'_1$, the composite map
    $$ \widetilde{\varphi}_A: \Q[\xx_n] \to \Q[\xx_A] \otimes \Q[\xx_{[n] \hspace{.5mm} \setminus \hspace{.5mm}  A}] \to R_{\lambda'_1} \otimes  R_{n,\widehat{\lambda},s}$$
    given by the canonical projection in both tensor components descends to a well-defined map $\varphi_A : R_n \to R_{\lambda'_1} \otimes R_{n,\widehat{\lambda},s}$.
\end{prop}

\begin{proof}
    It suffices to check that $\widetilde{\varphi}_A(I\nls) = 0$ in $R_{\lambda'_1} \otimes R_{n,\widehat{\lambda},s}$ for all $A \subset [n]$ with $|A| = \lambda_1'$. Since $s \geq \lambda_1'$, we have that $x_i^s \in I_{\lambda_1'}$ for $i \in A$. Observe $x_j^s \in I_{n,\widehat{\lambda},s}$ for $j \in [n] \hspace{.5mm} \setminus \hspace{.5mm} A$ by definition. Therefore, we have that
    $$
    \widetilde{\varphi}_A(x_i^s) = 
    \begin{cases}
        x_i^s \otimes 1 & \text{if }i \in A \\
        1 \otimes x_i^s & \text{if }i \in [n] \hspace{.5mm} \setminus \hspace{.5mm} A
    \end{cases}
    $$
    always vanishes in $R_{\lambda'_1} \otimes R_{n,\widehat{\lambda},s}$.

    Next, consider $e_d(B) \in I\nls$, where $B\subset [n]$ and $d$ satisfies 
    \begin{equation}\label{eq: tanisaki-def}
    d > |B| - (\lambda'_{n-|B|+1} + \dots + \lambda'_n)
    \end{equation}
    Notice we can factor $e_d(B)$ along the sets $A \cap B$ and $B  \hspace{.5mm} \setminus \hspace{.5mm} A$ to yield
    \begin{equation}\label{eq: tensor-factor}
        \widetilde{\varphi}_A(e_d(B)) = \sum_{i=\max(0,d - |B \hspace{.5mm} \setminus \hspace{.5mm} A|)}^{|A \cap B|} e_i(A \cap B) \otimes e_{d-i}(B  \hspace{.5mm} \setminus \hspace{.5mm} A)
    \end{equation}
    We show for each summand, either $e_i(A \cap B) \in I_{\lambda_1'}$, or $e_{d-i}(B  \hspace{.5mm} \setminus \hspace{.5mm} A) \in I_{n,\widehat{\lambda},s}$.

    First, we consider the case where $A \subset B$. Then we have that $e_i(A \cap B) = e_i(A) \in I_{\lambda_1'}$ whenever $i > 0$. It remains to check the case where $i = 0$, which is only possible if $|B| > |A|$.  
    Making the substitution $\widehat{\lambda}'_i = \lambda_{i+1}'$, we have
    \begin{align*}
        |B  \hspace{.5mm} \setminus \hspace{.5mm} A| - (\widehat{\lambda}'_{n - |A|- |B \hspace{.5mm} \setminus \hspace{.5mm} A| + 1 } + \dots )
        &= |B| - |A| - (\lambda'_{n - |B| + 2} + \dots) \\
        &\leq |B| - (\lambda'_{n-|B| + 1} + \dots ) < d
    \end{align*}
    where the weak inequality uses the fact that $|A| = \lambda'_1 \geq \lambda'_{n-|B|+1}$. This implies that $e_d(B  \hspace{.5mm} \setminus \hspace{.5mm} A) \in I_{n,\widehat{\lambda},s}$.

    Now consider the case where $A \not\subset B$. We will show that $e_{d-i}(B  \hspace{.5mm} \setminus \hspace{.5mm} A) \in I_{n,\widehat{\lambda},s}$ for all $i$. We have by \eqref{eq: tensor-factor} that $i \leq |A \cap B|$. Subtracting this from \eqref{eq: tanisaki-def}, we have
    \begin{equation}\label{eq: d-i-ineq}
        |B| - |A \cap B| - (\lambda_{n - |B| +1} + \dots) < d-i
    \end{equation}
    Then, as before, we have
    \begin{align*}
        |B  \hspace{.5mm} \setminus \hspace{.5mm} A| - (\widehat{\lambda}'_{n - |A|- |B \hspace{.5mm} \setminus \hspace{.5mm} A| + 1 } + \dots ) &= |B| - |A \cap B| - (\lambda'_{n-|B|+2-(|A|-|A\cap B|)} + \dots) \\
        &\leq |B| - |A \cap B| - (\lambda'_{n-|B|+1} + \dots) < d-i
    \end{align*}
    where the weak inequality uses $|A| - |A \cap B| > 0$ since $A \not\subset B$. This completes the proof.
\end{proof}

\begin{prop}\label{prop: descent-triangularity}
    Let $\ax \in \DD\nls$, and suppose
    \begin{equation}\label{eq: sum-des-mons-0}
        \sum_{\bx \leq_{\des} \ax} c_\bx x^\bx = 0
    \end{equation}
    in $R\nls$. Then, we have $c_\ax = 0$. In particular, the monomials $\BB\nls$ are linearly independent in $R\nls$.
\end{prop}

\begin{proof}
    We induct on the number of columns $\lambda_1$ of $\lambda$, i.e. the parts of $\lambda'$. If $\lambda_1 = 0$, i.e. $\lambda = \varnothing$, it is clear that $\BB_{n,\varnothing,s} = \{\xx^\ax: \ax \in [s-1]_0^n\}$ is a basis of $R_{n,\varnothing,s} = \Q[\xx_n]/(x_i^s: 1 \leq i \leq n)$, and are nonleading terms with respect to any total order on monomials.

    Now take $\lambda_1 \geq 1$, and let $\ax \in \DD\nls$, so that there is $\sigma = (\AAAA_\lambda\vert \BBB)\in \osp_{n,\lambda'}$ with $\ax \vert_{A_i} \in \DD_{\lambda'_i}$, $\ax\vert_{B_j} \in [s-1]_0$. Choose $A = A_1$ so $B = [n] \hspace{.5mm} \setminus \hspace{.5mm} A$, and let $\ax' = \ax|_A, \ax'' = \ax|_B$ denote the corresponding compositions in $\DD_{\lambda_1'},\DD_{n-\lambda_1',\widehat{\lambda},s}$, where $\widehat{\lambda}$ is as in the previous proposition. By Proposition \ref{prop: phi-delete-lambda1}, we compute
    \begin{equation}\label{eq: des-mons-expansion}
        0 = \varphi_A \bigg(\sum_{\bx \leq_{\des} \ax} c_\bx x^\bx\bigg) = \sum_{\bx' \in \DD_{\lambda_1'}} \xx^{\bx'} \otimes \bigg(\sum_{\bx''} d_{\bx',\bx''} \xx^{\bx''}\bigg)
    \end{equation}
    where we expand in the first factor using the descent basis, and push all coefficients into the second tensor factor. For each nonzero $d_{\bx',\bx''}$ in \eqref{eq: des-mons-expansion}, there is some $c_\bx \neq 0$ contributing to equation \eqref{eq: sum-des-mons-0} such that
    \begin{equation}
        \bx \leq_{\des} \ax  \quad \bx' \leq_{\des} \bx|_A, \quad \bx'' = \bx|_B, \quad
    \end{equation}
    where the second inequality follows from the fact that $\bx' \in \DD_{\lambda_1'}$ are the nonleading terms with respect to the $\des$-order (for an explicit straightening algorithm, see \cite{AllenDescent}).

    Since the left hand side of \eqref{eq: sum-des-mons-0} is $0$, we must have
    \begin{equation}
        \sum_{\bx''} d_{\ax',\bx''} \xx^{\bx''} = 0
    \end{equation}
    in $R_{n-\lambda_1',\widehat{\lambda},s}$. Furthermore, Lemma \ref{lem: descent-order-subset} implies that $d_{\ax',\bx''} \neq 0$ only if $\bx'' \leq_{\des} \ax''$, so we see $d_{\ax',\ax''} = c_\ax$ and write
    \begin{equation}
        \sum_{\bx''\leq_{\des} \ax''} d_{\ax',\bx''} \xx^{\bx'' } = 0
    \end{equation}
    The claim follows from applying the induction hypothesis.
\end{proof}

We can now prove Theorem \ref{thm: A}.

\begin{proof}[Proof of Theorem \ref{thm: A}]
    Proposition \ref{prop: descent-triangularity} shows that $\BB\nls$ are linearly independent, and by Corollary \ref{cor: count}, we have that $|\BB\nls| = \dim R\nls$. 
\end{proof}

Similar to \cite{carlsson2024descent}, we obtain the following corollary:

\begin{cor}
    There is an injection
    \begin{equation}\label{eq: injection-direct-sum}
    R\nls \hookrightarrow \bigoplus_{\sigma \in \osp_{n,\lambda}} R_{\lambda_1'} \otimes \dots \otimes R_{\lambda_m'} \otimes \Q[\xx_\BBB]/(x_i^s:i \in \BBB) =: M\nls
    \end{equation}
    obtained by patching together the maps $\varphi_{\lambda,A}$ and inducting on $\ell(\lambda')$.
\end{cor}

\section{Frobenius character of $R_{n,\lambda,s}$}\label{sec: Frob char}
Now, we use the descent basis of $R_{n,\lambda,s}$ to give a combinatorial formula for the graded Frobenius character $\Frob_q(R\nls)$ in terms of our indexing set $\Sigma\nls$.
First, recall the formula for the ungraded Frobenius character \eqref{eq: Frob Rnls wrt ssyt}. We can rewrite \eqref{eq: Frob Rnls wrt ssyt} in terms of the set  $\Sigma\nls$ by composing the bijection in Lemma \ref{lem: bij, battery and ssyt (no cc)} with the bijection $\Psi$ in Proposition \ref{prop: bij psi}:
\begin{align}\label{eq: ungraded Frob Rnls wrt Rnls}
    \Frob(R\nls) = \sum_{(S,\ix)\in \Sigma\nls} s_{\shape(S)}.
\end{align}
Now, we can use \eqref{eq: ungraded Frob Rnls wrt Rnls} to get $\dim(N_\gamma R\nls)$, which we will use to determine $\Hilb_q(N_\gamma R\nls)$. 
\begin{align*}
    \dim(N_\gamma R\nls) &= \langle e_\gamma, \Frob(R\nls) \rangle \\ &= \langle \omega e_\gamma, \sum_{(S,\ix)\in \Sigma\nls} \omega  \rangle  \\&= \sum_{(S,\ix)\in \Sigma\nls} \langle h_\gamma, s_{\shape(S)^t} \rangle  \\&= \sum_{(S,\ix)\in \Sigma\nls} K_{\shape(S)^t,\gamma} \numberthis \label{eq: dim antisym}
\end{align*}
where $ K_{\shape(S)^t,\gamma}$ is the Kostka number.
Using Corollary \ref{cor: kostka number permutations} and Proposition \ref{prop: RSK props} \ref{thm:RSK rev}, we can rewrite the last expression to be
\begin{align*}
    \dim(&N_\gamma R\nls) \\&= \sum_{(S,\ix)\in \Sigma\nls} | \{w\in \sym_n, P(w) = S^t, \Des(w)\subset \{\gamma_1,\gamma_1 + \gamma_2, \dots, \gamma_1 +\cdots + \gamma_{l-1}\}| \\ &= \sum_{(S,\ix)\in \Sigma\nls} | \{w\in \sym_n, P(w) = S, \text{Asc}(w)\subset \{\gamma_l,\gamma_l + \gamma_{l-1}, \dots, \gamma_l +\cdots + \gamma_{2}\}|.\numberthis \label{eq: dim N Rnls}
\end{align*}
where $\Asc(w)$ denotes the \emph{ascent set} of $w$, defined to be $\Asc(w) = \{i : w_i < w_{i+1}\}$.

Now, we construct a basis of $N_\gamma R\nls$ using the descent basis. This basis will give us an explicit formula for $\Hilb_q(R\nls)$, which in turn will uniquely determine $\Frob_q(R\nls)$.

For any composition $\gamma\models n$ and word $\mathbf{z}$ of length $n$, we say $\mathbf{z}$ is \emph{strictly increasing on $\gamma$} if when we divide $\mathbf{z}$ into blocks of sizes $\gamma_1,\gamma_2,\dots,\gamma_{l}$, we have that the entries within the blocks are strictly increasing.

\begin{example}
    Consider $\mathbf{z} = 2423679$ and $\gamma = (2,2,3)$. Then when we divide $\mathbf{z}$ into three blocks $24|23|679$, the entries are strictly increasing. Thus $\mathbf{z}$ is strictly increasing on $\gamma$.
\end{example}
We use this to define a subset of $\mathcal{D}\nls$.
\[\mathcal{D}\nls^{\gamma} := \{\ax \in \mathcal{D}\nls: \ax \text{ is strictly increasing on } \gamma \}.\]

We will show that this set indexes a basis of $N_\gamma R\nls$. We first relate this set to pairs $(w,\mu)$ where $(P(w),\mu)\in \Sigma\nls$. To do this, we make the following observation.

\begin{lem}\label{lem: asc}
    Consider $\ax\in \mathcal{D}\nls$ and $\cc^{-1}(\rev(\ax)) = (w,\mu)$. Then we have $a_{n-j}< a_{n-j+1}$ if and only if $w_j > w_{j+1}$.
\end{lem}

\begin{proof}
    Note that the condition $a_{n-j}<a_{n-j+1}$ is equal to $z_{j} > z_{j+1}$, where $\mathbf{z} = \rev(\ax) = \cc(w,\mu)$. By construction we can see that $ \cc(w,\mu)_j >  \cc(w,\mu)_{j+1}$ can only happen if $w_j >  w_{j+1}$.
    In particular: if $w_j > w_{j+1}$, it is clear that $\cc(w,\mu)_j >  \cc(w,\mu)_{j+1}$ since $\cc(w)_j > \cc(w)_{j+1}$. On the other hand, if $w_j < w_{j+1}$, then we have $\cc(w,\mu)_j \leq   \cc(w,\mu)_{j+1}$.
\end{proof}

Using Lemma \ref{lem: asc}, we can compute $|\mathcal{D}\nls^\gamma|$, by looking at the preimage of $\mathcal{D}\nls^\gamma$ under the map $(w,\mu)\mapsto \rev(\cc(w,\mu))$.

\begin{cor}\label{cor: subset of perms}
    For $\gamma\models n$, we have $\rev(\cc(w,\mu))\in \mathcal{D}^{\gamma}\nls$ if and only if $ \Asc(w) \subset \{\gamma_l,\gamma_{l}+\gamma_{l-1},\dots, \gamma_l+\gamma_{l-1}+\cdots+\gamma_{2}\}$.
    In particular, we have $|\mathcal{D}\nls^{\gamma}| = \dim(N_\gamma R\nls)$.
\end{cor}

\begin{proof}
The first statement follows immediately from Lemma \ref{lem: asc}. The second statement follows from \eqref{eq: dim N Rnls}.
\end{proof}

We can also use the following fact about $\mathcal{D}\nls$. Note that this is analogous to \cite[Lemma 3.4]{carlsson2024descent} and the same proof holds. 

\begin{lem}\label{lem: swapping}
      Consider $\ax\in \mathcal{D}\nls$ with $a_{j} > a_{j+1}$. Let $\tilde{\ax}$ be the word we get by swapping $a_{j}$ and $a_{j+1}$. Then $\tilde{\ax}\in \mathcal{D}\nls$.
\end{lem}

Now, we can use these results to construct an explicit basis of $N_\gamma R\nls$.
\begin{prop}\label{prop: antisym basis}
   $\mathcal{B}\nls^\gamma: = \{N_\gamma \xx^{\ax}: \ax\in \mathcal{D}\nls^{\gamma}\}$ is a basis of $N_\gamma R\nls$.
\end{prop}

\begin{proof}
From Theorem \ref{thm: A}, we know that $\{N_\gamma \xx^{\ax}: \ax\in \mathcal{D}\nls\}$ must span $N_\gamma R\nls$. 

We show that $\mathcal{B}\nls^\gamma$ spans $R\nls$. 
First, consider $\ax\in \mathcal{D}\nls$ such that when we split $\ax$ into blocks of size $\gamma_1,\dots, \gamma_l$, some block contains a repeated entry. Then, we must have $N_\gamma \xx^{\ax} = 0$, since $N_\gamma$ antisymmetrizes the exponents within the blocks. 

Now, consider $\ax\in \mathcal{D}\nls$ where the blocks consist of distinct entries, but $\ax\not\in \mathcal{D}\nls$. This means there exists an index $j$ such that $a_j> a_{j+1}$ belong to the same block. If $\tilde{ax}$ is the result of swapping $a_j$ and $a_{j+1}$, we have that $N_\gamma \xx^{\ax} = -N_\gamma \xx^{\tilde{\ax}}$. Furthermore, we know $\tilde{\ax} \in \mathcal{D}\nls$ by Lemma \ref{lem: swapping}.  In this way, for each such $\ax$, we can construct a $\bx\in \mathcal{D}\nls^\gamma$ such that $N_\gamma \xx^{\ax} = \pm N_\gamma \xx^{\bx}$. Thus $\mathcal{B}\nls^\gamma$ spans $N_\gamma R\nls$.

Note that $|\mathcal{B}\nls^\gamma| = |\mathcal{D}\nls^\gamma|$, thus from Corollary \ref{cor: subset of perms} we have that $\mathcal{B}\nls^\gamma$ is a basis of $N_\gamma R\nls$.
\end{proof}

\begin{cor}
    We have 
    \begin{equation}\label{eq;hilb antisym}
        \Hilb_q(N_\gamma R\nls) = \sum_{\substack
        {(w,\mu) \\ 
        (P(w), \mu) \in \Sigma\nls\\ 
        \Asc(w) \subset 
        \{\gamma_l, \gamma_l+\gamma_{l-1},\dots, \gamma_l+\gamma_{l-1}+\cdots+\gamma_{2} \} 
        }
        }
        q^{\cocharge(w)+|\mu|}.
    \end{equation}

\end{cor}
\begin{proof}
    This follows from the fact that the degree of $N_\gamma \xx^{\ax}$ is equal to $\cocharge(w)+|\mu|$, where $\rev(\cc(w,\mu)) = \ax$, and Proposition \ref{prop: antisym basis}.
\end{proof}

Finally, we get an explicit Schur expansion of $\Frob_q(R\nls)$ using $\Sigma\nls$.

\begin{thmb}\makeatletter\def\@currentlabel{B}
\begin{align}\label{eq: Frob, Sigmanls}
    \Frob_q(R\nls) = \sum\limits_{(S,\mathbf{i})\in \Sigma\nls} q^{\cocharge(S) + |\mu|} s_{\shape(S)}
\end{align}
\end{thmb}

\begin{proof}
  Doing the same computation as in \eqref{eq: dim antisym} and \eqref{eq: dim N Rnls}, except now carrying over the power of $q$, we have
    \begin{align*}
        \langle e_\gamma, \sum\limits_{(S,\mu)\in \Sigma\nls} q^{\cocharge(S) +|\mu|} s_{\shape(S)}\rangle  &=   \sum_{(S,\mu)\in \Sigma\nls} q^{\cocharge(S) +|\mu|}  K_{\shape(S)^t,\gamma}  \\  &= \sum\limits_{\substack{(w,\mu) \\ (P(w),\mu) \in \Sigma\nls\\  \text{Asc}(w)\subset \{\gamma_l,\gamma_l + \gamma_{l-1}, \dots, \gamma_l +\cdots + \gamma_{2}\}}} 
        q^{\cocharge(P(w)) +|\mu|}.
        \numberthis\label{eq: final line} \end{align*}
Since cocharge$(P(w)) = \cocharge(w)$, we can see that \eqref{eq: final line} is equal to \eqref{eq;hilb antisym}.
\end{proof}

\subsection{Connections to known formulas for Frobenius characters}
We can see that the correct specializations recover the known Frobenius characters for $R_{\lambda}$ and $R_{n,k}$. In particular, when $(n,\lambda,s) = (k,\lambda,\ell(\lambda))$, we recover \eqref{eq: frob R_lambda}, as proven in \cite{hanada2025charge}.

When we set $(n,\lambda,s) = (n,1^k,k)$, we get
  \[\Frob_q(R\nls) = \sum\limits_{\substack{S\in \SYT_n\\ \des(S) < k}} \left(\sum\limits_{\mu\subset(n-k)\times(k-\des(S)-1)} q^{|\mu|}\right) q^{\cocharge(S)} s_{\shape(S)}.\]

Note that this is equivalent to following expression for $\Frob_q(R_{n,k})$ \cite[Corollary 6.13]{HaglundRhoadesShimozono2018}, using properties of the $q$-binomial coefficients, as well Proposition \ref{prop: maj, cc}.
\[\Frob_q(R_{n,k}) = \sum\limits_{\substack{S\in \SYT_n\\ \des(S) < k}} \stirling{n-\des(S)-1}{n-k}_{\!q} q^{\maj(S)} s_{\shape(S)}.\]

Furthermore, there is the following Schur expansion formula for $\Frob_q(R\nls)$, due to Gillespie--Griffin \cite{battery}, in terms of battery-powered tableaux. 
\begin{thm}[Gillespie--Griffin {\cite[Theorem 1.6]{battery}}]\label{thm: og battery}
    \begin{align}\label{eq: Frob, battery}
        \Frob_q(R\nls) = \frac{1}{q^{\binom{s-1}{2}(n-k)
}} \sum\limits_{T = (D,B) \in \mathcal{T}^+\nls} q^{\cocharge(T)} s_{\shape(D)}.
\end{align}
\end{thm}
In \cite{battery}, the authors give a geometric proof of \eqref{eq: Frob, battery} and pose the question whether there is more direct algebraic or combinatorial proof. Though recent results \cite{skewing_rise_delta} give such explanation for $R_{n,k}$, there is still no known such argument for general $R\nls$.

Our formula provides an algebraic and combinatorial proof for Theorem \ref{thm: og battery} by showing that it is equivalent to Theorem \ref{thm: b} up to a combinatorial bijection.

\begin{corc}
Theorem \ref{thm: b} is equivalent to Theorem \ref{thm: og battery} up to a combinatorial bijection between the sets $\Sigma\nls$ and $\mathcal{T}^+\nls$. 
\end{corc}

\begin{proof}
The equality of \eqref{eq: Frob, battery} and  \eqref{eq: Frob, Sigmanls} follows from Proposition \ref{prop: bij psi}. In particular, if we apply $\Psi$ to \eqref{eq: Frob, Sigmanls}, we get \eqref{eq: Frob, battery}. 
\end{proof}

\section{Higher Specht Basis}\label{sec:specht}
We now give our second basis of $R\nls$, which is a \emph{higher Specht basis}, generalizing the work of Gillespie-Rhoades for $R_{n,k}$ \cite{RhoadesGillespie2021}, which was in turn an extension of Ariki-Terasoma-Yamada \cite{ariki1997higher}. Most of the results in this section hold for arbitrary $\lambda$; we will specify when we invoke that $\lambda$ has at most two parts. Our construction is obtained from generalizing the work of the second author in \cite{hanada2025charge}. Recall that here we are taking our partitions $\mu\subseteq (n-k)\times (k-\des(S)-1)$ and writing them as tuples $(i_1,\dots, i_{n-k})$ of nonnegative integers, where $i_{n-k} = \mu_{n-k}, i_{j} = \mu_j - \mu_{j+1}$.

\begin{defn}\label{def: higherspechtbasis}
    Let $\lambda \vdash k$, $k \leq n$, and $\ell(\lambda) \leq s$. We define
    \begin{equation}\label{eq: def-higherspechtbasis}
        \CC\nls := \{ F_T^S e_1^{i_1}\dots e_{n-k}^{i_{n-k}}: (S,\ix)\in \Sigma\nls, \shape(S) = \shape(T)\}
    \end{equation}
\end{defn}


\begin{thmc}\makeatletter\def\@currentlabel{D}\label{thm: higherspecht}
    Let $\lambda = (\lambda_1,\lambda_2)$ be a partition with at most two rows. The collection of polynomials $\CC_{n,\lambda,\ell(\lambda)}$ forms a higher Specht basis of $R_{n,\lambda,\ell(\lambda)}$.
\end{thmc}

The proof of Theorem \ref{thm: higherspecht} will involves showing $\partial f \cdot g = 0$ for certain polynomials $f,g \in \Q[\xx_n]$. To do this, we first expand $f,g$ in terms of monomials:
$$ f = \sum_\ax c_\ax \xx^\ax \qquad g = \sum_\bx d_\bx \xx^\bx$$
so that by linearity, we have

\begin{equation}\label{eq: partials}
 \partial f \cdot g = \sum_{\ax,\bx} c_\ax d_\bx (\partial \xx^\ax \cdot \xx^\bx).
\end{equation}
We then construct a sign-reversing involution on the terms. The following two subsections will give a combinatorial mechanism for keeping track of the terms. The final subsection will contain the proof of Theorem \ref{thm: higherspecht}.

\subsection{Tableau Formula for higher Specht Polynomials}

We now give a combinatorial formulation for higher Specht polynomials as a signed sum over certain fillings of a tableau.

First, we go over notation we will use throughout this section. 
Let $\gamma \subset \Z_{\geq 0} \times \Z_{\geq 0}$ to be a \emph{diagram} in the plane. We use $\mu,\lambda$ for partition shapes, and reserve $\gamma$ for more irregular shapes. A \emph{filling} of $\gamma$ is an assignment of nonnegative integers to each box $D: \gamma \to \N$. For a box $u = (a,b) \in \gamma$, we denote its entry by $D(u)$, and set the \emph{height} to be $\height(u) = b$. For a one-to-one filling, we may define the inverse map $D^{-1}: \im(D) \to \gamma$.

\begin{defn}\label{def: hs-monomial}
    Let $S,P$ denote fillings of the same shape $\lambda$, where $P$ has content $(1^n)$. Then, define the monomial
    $$ \xx^S_P := \prod_{u \in \lambda} x_{P(u)}^{S(u)}.$$
\end{defn}

\begin{defn}\label{def: T-snaking}
    Let $T \in \SYT(\lambda)$. A \emph{$T$-snaking} is a filling $P$ of $\lambda$ with content $1^n$ such that if $i,j$ occur in the same column in $T$, then $i,j$ occur in different rows of $P$. We refer to the set of $T$-snakings as $\snake(T)$.
\end{defn}

We refer to these diagrams as $T$-snakings, because one can imagine taking the columns of $T$ and ``snaking" them down $\lambda$ in some order.

Let $T_i = \{a_1,\dots,a_m\}$ denote the set of entries in the $i$th column of $T$. Suppose these appear in the order $a_{\tau(1)},\dots,a_{\tau(m)}$ from bottom to top in $P$ where $\tau^{(i)}(P)$ is a permutation. We refer to the boxes $u_1,\dots,u_m \in \lambda$ corresponding to $a_1,\dots,a_m$ as the \emph{$i$th snake} of $P$, denoted $P^{(i)}$, and set
$$ (-1)^P := \prod_{i} (-1)^{\tau^{(i)}(P)}$$

\begin{example}\label{ex: T-snaking}
    Consider the following tableaux:
    $$S =
    \begin{ytableau}
        2 & 3 \\
        1 & 1 & 2 \\
        0 & 0 & 1 & 3
    \end{ytableau}
    , \qquad 
    T = 
    \begin{ytableau}
        *(pink)5 & *(cyan)9 \\
        *(pink)3 & *(cyan)6 & *(lime)8 \\
        *(pink)1 & *(cyan)2 & *(lime)4 & *(magenta)7
    \end{ytableau}
    $$
    Then, an example of a $T$-snaking $P$, with its corresponding signed monomial $(-1)^P\xx_P^S$ is given by
    $$P=
    \begin{ytableau}
        *(cyan)2 & *(pink)1 \\
        *(lime)4 & *(pink)5 & *(cyan)6 \\
        *(lime)8 & *(magenta)7 & *(cyan)9 & *(pink)3 
    \end{ytableau}
    \qquad
    (-1)^P\xx_P^S = +x_1^3 x_2^2 x_3^3 x_4^1 x_5^1 x_6^2 x_7^0 x_8^0 x_9^1 = +x_1^3 x_2^2 x_3^3 x_4 x_5 x_6^2 x_9
    $$
    $$
    (-1)^P = (-1)^{|231|}(-1)^{|321|}(-1)^{|21|}(-1)^{|1|} = (-1)^0(-1)^1(-1)^1(-1)^0 = +1
    $$
    where we have colored each column its corresponding snake accordingly.
\end{example}

\begin{lem}\label{lem: snaking-is-hs}
    Let $S$ be a $\lambda$-filling, and $T \in \SYT(\lambda)$. Then, the \emph{higher Specht polynomial} $F_T^S$ is given by
    \begin{equation}\label{eq: t-snakings}
        F_T^S = \varepsilon_T x_T^S = \sum_{P \in \snake(T)} (-1)^P \xx_P^S.
    \end{equation}
\end{lem}

\begin{example}
    Let $S = \ytableaushort{11,00}, T = \ytableaushort{34,12}$. The sixteen $T$-snakings, along with their monomials and signs are
    \begin{align*}
        F_T^S &=
        &\begin{ytableau}
            *(pink)3 & *(cyan)4 \\
            *(pink)1 & *(cyan)2
        \end{ytableau}\quad
        &
        &\begin{ytableau}
            *(pink)1 & *(cyan)4 \\
            *(pink)3 & *(cyan)2
        \end{ytableau}\quad
        &
        &\begin{ytableau}
            *(pink)3 & *(cyan)2 \\
            *(pink)1 & *(cyan)4
        \end{ytableau}\quad
        &
        &\begin{ytableau}
            *(pink)1 & *(cyan)2 \\
            *(pink)3 & *(cyan)4
        \end{ytableau}\quad
        &
        &\begin{ytableau}
            *(pink)3 & *(cyan)4 \\
            *(cyan)2 & *(pink)1
        \end{ytableau}\quad
        &
        &\begin{ytableau}
            *(pink)1 & *(cyan)4 \\
            *(cyan)2 & *(pink)3
        \end{ytableau}\quad
        &
        &\begin{ytableau}
            *(pink)3 & *(cyan)2 \\
            *(cyan)4 & *(pink)1
        \end{ytableau}\quad
        &
        &\begin{ytableau}
            *(pink)1 & *(cyan)2 \\
            *(cyan)4 & *(pink)3
        \end{ytableau}\quad
        \\
        &
        &x_3x_4 \quad
        &-&x_1x_4\quad
        &-&x_2x_3\quad
        &+&x_1x_2\quad
        &+&x_3x_4\quad
        &-&x_1x_4\quad
        &-&x_2x_3\quad
        &+&x_1x_2\quad \\
        &
        &\begin{ytableau}
            *(cyan)4 & *(pink)3\\
            *(cyan)2 & *(pink)1
        \end{ytableau}\quad
        &
        &\begin{ytableau}
            *(cyan)4 & *(pink)1\\
            *(cyan)2 & *(pink)3
        \end{ytableau}\quad
        &
        &\begin{ytableau}
            *(cyan)2 & *(pink)3\\
            *(cyan)4 & *(pink)1
        \end{ytableau}\quad
        &
        &\begin{ytableau}
            *(cyan)2 & *(pink)1\\
            *(cyan)4 & *(pink)3
        \end{ytableau}\quad
        &
        &\begin{ytableau}
            *(cyan)4 & *(pink)3\\
            *(pink)1 & *(cyan)2
        \end{ytableau}\quad
        &
        &\begin{ytableau}
            *(cyan)4 & *(pink)1\\
            *(pink)3 & *(cyan)2
        \end{ytableau}\quad
        &
        &\begin{ytableau}
            *(cyan)2 & *(pink)3\\
            *(pink)1 & *(cyan)4
        \end{ytableau}\quad
        &
        &\begin{ytableau}
            *(cyan)2 & *(pink)1\\
            *(pink)3 & *(cyan)4
        \end{ytableau}\quad
        \\
        &+
        &x_3x_4\quad
        &-&x_1x_4\quad
        &-&x_2x_3\quad
        &+&x_1x_2\quad
        &+&x_3x_4\quad
        &-&x_1x_4\quad
        &-&x_2x_3\quad
        &+&x_1x_2\quad
    \end{align*}
\end{example}

\subsection{Constructions for Injective Tableau}

We first recall a few definitions and theorems from Rhoades-Yu-Zhao \cite{rhoades2020harmonic}, which will be key ingredients in our proof. We adjust notations slightly to match Griffin's for consistency. 

\begin{defn}\label{def: Vnls}
    The \emph{harmonic space} $V\nls$ is defined to be the Macaulay inverse system of $R\nls$:
    $$ V\nls := I\nls^\perp = \{g \in \Q[\xx_n] : \partial f \cdot g = 0, \forall f \in I\nls\} $$
\end{defn}

Given a $\lambda$-filling $U: \lambda\to \N$ with positive entries, we say that $U$ is \emph{column strict} if the entries decrease in the columns from top to bottom. We say that $U$ is \emph{injective} if $U$ contains distinct entries. The set of column-strict, injective fillings with entries $\leq n$ is denoted by $\Inj(\lambda, \leq n)$. 

Let $A = \{i_1 < \dots < i_r\} \subset [n]$ be a size $r$-subset. We may consider the Vandermonde determinant in these variables:

$$ \Delta(A) := 
\begin{vmatrix}
    1 & \dots & 1 \\
    x_{i_1} & \dots & x_{i_r} \\
    \vdots & \ddots & \vdots\\
    x_{i_1}^{r-1} & \dots& x_{i_r}^{r-1}
\end{vmatrix}
=
\varepsilon_r \cdot (x_{i_2}x_{i_3}^2\dots x_{i_r}^{r-1})
$$
where $\varepsilon_r := \sum_{\sigma \in \sym_r} (-1)^\sigma \cdot \sigma$ is the sign idempotent for $\sym_r$.

Given a $\lambda$-filling $U$, where $\lambda$ has $m$ columns, let the sets of entries in each column be denoted by $U_1,\dots,U_m$. Then, we define

$$ \Delta_U := \Delta(U_1)\dots \Delta(U_m)$$
to be the products of Vandermondes in each column. Rearranging the entries within the columns of $U$ only changes the sign of $\Delta_U$, so we typically assume $U$ is column strict. In this case, we have that $\Delta_U$ is the \emph{Garnir polynomial} associated to $U$. 

If $U$ is injective as well with entries $\leq n$, Rhoades-Yu-Zhao (\cite{rhoades2020harmonic}) introduce an $(n,\lambda,s)$-extension as follows:

$$ \Delta_{U,s} := \Delta_U \prod_{\substack{1 \leq i \leq n \\ i \not\in U}} x_i^{s-1}.$$

We can now describe the generating set of $V\nls$:

\begin{thm}[Rhoades-Yu-Zhao {\cite{rhoades2020harmonic}}]\label{thm: harmonic-space-gens}
    Fix $s \geq \ell(\lambda)$. Then, we have

    $$V\nls = \Q[\partial \xx_n] \cdot \{\Delta_{U,s} : U \in \Inj(\lambda,\leq n)\}$$
\end{thm}

\begin{cor}\label{cor: ideal-iff-kills-garnirs}
    We have that $f \in I\nls$ if and only if $\partial f \cdot \Delta_{U,s} = 0$ for all $T \in \Inj(\lambda,\leq n)$.
\end{cor}

We now give a combinatorial representation for the monomials of $\Delta_{U,s}$.

\begin{defn}\label{def: Gamma-monomial}
    Let $D$ denote a filling of a diagram $\gamma \subset \Z_{\geq 0} \times \Z_{\geq 0}$ in the plane. We denote the \emph{monomial} of $D$, written $\xx_D$, to be

    $$ \xx_D := \prod_{u \in \gamma} x_{D(u)}^{\height(u)}$$
\end{defn}

We define a $i$th-\emph{column} of $\gamma$ to be the collection of (not necessarily connected) boxes $u = (i,b)$ with fixed $x$-coordinate for arbitrary $b$. If the entries in each column of $D$ are distinct, let $D_i := \{D(u) : u = (i,b)\}$ denote the entries in each column, and let $\sigma^{(i)}$ be the standardization (from bottom to top) of the $i$th column. We set $(-1)^D:= (-1)^{\sigma^{(1)}}\dots(-1)^{\sigma^{(m)}}$.

Fix $\lambda \vdash k \leq n$, and $s \geq \ell(\lambda)$. Set $\xi = \{ (\lambda_1+j,s-1):1 \leq j \leq n-k\}$ to be a horizontal strip of length $n-k$ with height $s-1$ starting at $\lambda_1+1$, and set $\gamma = \lambda \sqcup \xi$. Given an injective tableau $U \in \Inj(\lambda,\leq n)$, define
$$ \DX_{\col(U,s)} := \{ D: \gamma \to \N \ \vert \  D_i = U_i, \{D(u):u \in \xi\} =  [n]  \hspace{.5mm} \setminus \hspace{.5mm} (U_1 \cup\dots\cup U_m), \xi \text{ increasing} \} $$
i.e. the $\lambda$ component is a rearrangement of the entries within the columns of $U$, and the remaining entries fill $\xi$ in increasing order.

\begin{example}
    Let $\lambda = (3,1)$, $n = 7$, $s = 4$, so that $n-k = 3$. Then, we have
    $$ 
    \gamma =
    \begin{ytableau}
        \none & \none & \none & 3 & 3 & 3 \\
        \none \\
        1 \\
        0 & 0 & 0
    \end{ytableau}
    $$
    where the number in each box denotes its height. The three boxes in a row containing the $3$'s correspond to $\xi$.
\end{example}

\begin{lem}\label{lem: garnir-tab-formulation}
    We have that
    \begin{equation}\label{eq: garnir-tab-formulation}
        \Delta_{U,s} = \sum_{D \in \DX_{\col}(U,s)} (-1)^{D} \xx_D
    \end{equation}
\end{lem}

\begin{example}
    Let $\lambda = (2,1,1), n = 5, s = 4$, let $U = \ytableaushort{5,3,12}$ be an injective tableau. We have
    \begin{align*}
        \Delta_{U,s} &=
        &\begin{ytableau}
            \none & \none & 4 \\
            5 \\
            3 \\
            1 & 2
        \end{ytableau}\quad
        &&\begin{ytableau}
            \none & \none & 4 \\
            3 \\
            5 \\
            1 & 2
        \end{ytableau}\quad
        &&\begin{ytableau}
            \none & \none & 4 \\
            5 \\
            1 \\
            3 & 2
        \end{ytableau}\quad
        &&\begin{ytableau}
            \none & \none & 4 \\
            1 \\
            5 \\
            3 & 2
        \end{ytableau}\quad
        &&\begin{ytableau}
            \none & \none & 4 \\
            3 \\
            1 \\
            5 & 2
        \end{ytableau}\quad
        &&\begin{ytableau}
            \none & \none & 4 \\
            1 \\
            3 \\
            5 & 2
        \end{ytableau}\quad
        \\
        &&x_3x_5^2x_4^3 \quad
        &-&x_3^2 x_5 x_4^3\quad
        &-&x_1x_5^2x_4^3\quad
        &+&x_1^2x_5x_4^3\quad
        &+&x_1x_3^2x_4^3\quad
        &-&x_1^2x_3x_4^3\quad
    \end{align*}
\end{example}

Now we discuss how to keep track of the terms in \eqref{eq: partials}. For our applications, the expression $\partial f \cdot g$ will typically have $f = F_T^S, g = \Delta_{U,s}$ for the appropriate fillings $S,T,U$. As such, the monomials of $f$ will correspond to $T$-snakings $P \in \snake(T)$, and the monomials of $g$ will correspond to fillings $D \in \DX_{\col}(U,s)$. We conduct an analysis on the monomial $\partial \xx_P^S \cdot \xx_D$.

Define the exponent tuples to be
\begin{equation}\label{eq: exponent-tuple-def}
\ax(P,S) := ((S \circ P^{-1})(i))_{1 \leq i \leq n} \qquad \bx(D) := ((\height \circ D^{-1})(i))_{1 \leq i \leq n}
\end{equation}
That is, we have  $\ax(P,S)_i$ (resp. $\bx(D)_i$)
is equal to the exponent of $x_i$ in the monomial $\xx_P^S$ (resp. $\xx_D$). The following lemma is immediate.

\begin{lem}\label{lem: derivative-def}
    If $\ax(P,S)_i > \bx(D)_i$ for any $1 \leq i \leq n$, then $\partial \xx_P^S \cdot \xx_D = 0$. Otherwise,
    $$ \partial \xx_P^S \cdot \xx_D = \partial \xx^{\ax} \cdot \xx^{\bx} = C\prod_{i = 1}^n x_i^{\bx(D)_i - \ax(P,S)_i} $$
    where $C = \prod_{1 \leq i \leq n} \frac{\bx(D)_i!}{(\bx(D)_i-\ax(P,S)_i)!}$.
\end{lem}

\begin{example}\label{ex: derivatives}
    Let $P,S,T$ be as in Example \ref{ex: T-snaking}:
    $$
    S =
    \begin{ytableau}
        2 & 3 \\
        1 & 1 & 2 \\
        0 & 0 & 1 & 3
    \end{ytableau}
    , \qquad 
    T = 
    \begin{ytableau}
        *(pink)5 & *(cyan)9 \\
        *(pink)3 & *(cyan)6 & *(lime)8 \\
        *(pink)1 & *(cyan)2 & *(lime)4 & *(magenta)7
    \end{ytableau}
    , \qquad 
    P=
    \begin{ytableau}
        *(cyan)2 & *(pink)1 \\
        *(lime)4 & *(pink)5 & *(cyan)6 \\
        *(lime)8 & *(magenta)7 & *(cyan)9 & *(pink)3 
    \end{ytableau}
    $$
    then we have $\ax(P,S) = (3,2,3,1,1,2,0,0,1)$, so $\xx^\ax = x_1^3x_2^2x_3^3x_4x_5x_6^2x_9$, as before.

    Let $s = 3$, and $U, D \in \DX_{\col}(U,s)$ be as follows:
    $$U = 
    \begin{ytableau}
        7 \\
        4 & 8 \\
        2 & 5
    \end{ytableau}
    , \qquad 
    D = 
    \begin{ytableau}
        \none & \none & 1 & 3 & 6 & 9 \\
        2 \\
        4 & 5 \\
        7 & 8
    \end{ytableau}
    $$
    so that $\bx(D) = (3,2,3,1,1,3,0,0,3)$ and $\xx^\bx = x_1^3x_2^2x_3^3x_4x_5x_6^3x_9^3$. We have that 
    $$\partial \xx^\ax \cdot \xx^\bx = 3!\cdot2!\cdot3!\cdot 3 \cdot 3 \cdot x_6 x_9^2$$
    and indeed, $\ax_i \leq \bx_i$ for $1 \leq i \leq n$.
\end{example}

\subsection{Proof of Basis}

We now prove for $\lambda = (\lambda_1,\lambda_2)$ and $s = \ell(\lambda)$ that $\CC\nls$ is a basis $R\nls$: we give a brief outline of the argument here for clarity. First, we prove that for all $\lambda \vdash k \leq n$ and $s = \ell(\lambda)$, there is a composite surjective map

$$ \varphi: \Q[x_1,\dots,x_n]/(x_1^s,\dots,x_n^s) \twoheadrightarrow R_{n,\nu,s} \twoheadrightarrow R_{n,\lambda,s}$$
where $\nu$ is the unique smallest shape in dominance order with $\ell(\nu) = s$. Then, we show that 
\[\widetilde{\CC}_{n,s} = \{ F_T^S e_1^{i_1}\dots e_{n-k}^{i_{n-k}} : (S,\ix)\in \Sigma_{n,1^k, s}, \shape(S) = \shape(T)\}\]
spans $R_{n,\nu,s}$. We then show that if $\varphi(f) \neq 0$ for $f\in \tilde{\mathcal{C}}_{n,s} $, then $f \in \CC\nls$, so that $\CC\nls$ spans $R\nls$. Combined with Corollary \ref{cor: count}, we have that $\CC\nls$ is a basis of $R\nls$.

\begin{lem}\label{lem:containment oif ideals}
    Let $\lambda, \nu$ be as above, and $s = \ell(\lambda)$. We have the containment of ideals
    $$ (x_1^s,\dots,x_n^s) \subset I_{n,\nu,s} \subset I\nls$$
    As a consequence, the surjection $\varphi$ factors through $R_{n,\nu,s}$:
    $$
    \varphi: \Q[x_1,\dots,x_n]/(x_1^s,\dots,x_n^s) \twoheadrightarrow R_{n,\nu,s} \twoheadrightarrow R_{n,\lambda,s}
    $$
\end{lem}

\begin{proof}
    The first containment is by definition. Let $e_d(S) \in I_{n,\nu,s}$, so that $d > |S| - \nu'_n - \dots - \nu'_{n-|S|+1}$. Since $\nu \trianglelefteq \lambda$, we have that $\nu' \trianglerighteq \lambda'$, so that
    \begin{equation}\label{eq: dominance}
     \nu'_1 + \dots + \nu'_{n-|S|} \geq \lambda'_1 + \dots + \lambda'_{n-|S|}
     \end{equation}
    adding $|S|-k$ to both sides, and combining with \ref{eq: dominance}, we have that
    $$d > |S| - \nu'_n - \dots - \nu'_{n-|S|+1} \geq |S| - \lambda'_n - \dots - \lambda_{n-|S|+1}'$$
    which completes the proof.
\end{proof}

\begin{lem}[Gillespie--Rhoades \cite{RhoadesGillespie2021}]\label{lem: augmented-span}
    The augmented basis
    \begin{align}\label{eq:big basis}
        \{ F_T^S e_1^{i_1}\dots e_{n}^{i_{n}}: (S,\ix)\in \Sigma_{n,\emptyset, s}, \shape(S) = \shape(T) \}
    \end{align}
    is a basis of $\Q[\xx_n]/(x_1^s,\dots,x_n^s)$ for any $s \geq 0$.
\end{lem}

We will need the following extension of $I_{n,k} := I_{n,(1^k),k}$, first defined in \cite{HaglundRhoadesShimozono2018}:

\begin{defn}[Haglund--Rhoades--Shimozono {\cite{HaglundRhoadesShimozono2018}}]\label{def: Ink-smaller-k}
    Suppose $0 \leq s,k \leq n$. Let
    $$ I_{n,k,s} := (x_1^s,\dots, x_n^s, e_n(\xx_n),\dots,e_{n-k+1}(\xx_n)).$$
    Note we allow $s < k$ here.
\end{defn}

\begin{lem}[Rhoades--Yu--Zhao {\cite{rhoades2020harmonic}}]\label{lem: s-small-ideal-equal}
    Suppose $s < k$, and write $k = qs + r$ for integers $q,r \geq 0$. Let $\nu = ( (q+1)^r,q^{s-r})$. Then, we have that  $I_{n,k,s} = I_{n,\nu,s}$.
\end{lem}

The following result is a corollary of Lemma \ref{lem: s-small-ideal-equal}.

\begin{cor}\label{cor: lambda-tilde-span}
    The set 
    \begin{align}\label{eq: basis_Rnnus}
        \widetilde{\CC}_{n,s} = \{ F_T^S e_1^{i_1}\dots e_{n-k}^{i_{n-k}} : (S,\ix)\in \Sigma_{n,1^k, s}, \shape(S) = \shape(T)\}
    \end{align}
    spans $R_{n,\nu,s}$. Furthermore, $\nu$ is the unique smallest partition in dominance order such that $\nu \trianglelefteq \lambda$ and $\ell(
    \nu) = s$.
\end{cor}

\begin{proof}
    Note that $e_j(\xx_n)\in I_{n,k,s}$ for $n-k+1\leq j\leq n$,  thus $\tilde{\mathcal{C}}_{n,s}$ spans $R_{n,k,s}$ by Lemma \ref{lem: augmented-span}. Thus the first statement is apparent from Lemma \ref{lem: s-small-ideal-equal}. To prove the second statement, consider $\lambda',\nu'$, and note $\ell(\lambda) = \ell(\nu) = s$ translates to $s = \lambda'_1 = \nu'_1$, so that all parts of $\lambda',\nu'$ are no greater than $s$. Then, if 
    $$ \lambda_1' + \dots + \lambda_j' > \nu_1' + \dots + \nu'_j = sj $$
    then $\lambda_i' > s$ for some $i$, a contradiction. Therefore, $\lambda' \trianglelefteq \nu'$, so $\nu \trianglelefteq \lambda$.
\end{proof}
   Note that $\widetilde{\mathcal{C}}_{n,s}\neq \mathcal{C}_{n,\nu,s}$ is not necessarily a basis.

\begin{prop}\label{prop: lambda-tilde-bad-death}
    Let $S \in \SYT_n$, $\lambda = (\lambda_1,\lambda_2)$ a partition of at most two rows, and $s \geq \ell(\lambda)$. If $\ctype(S\vert_k) \not\trianglerighteq \lambda$, then $F_T^S = 0$ in $R_{n,\lambda,s}$ for all $T \in \SYT_n$, where $\sh(S) = \sh(T)$.
\end{prop}

\begin{proof}
    Suppose $\ctype(S\vert_k) \not\trianglerighteq \lambda$. By Corollary \ref{cor: ideal-iff-kills-garnirs}, we need to show that $\partial F_T^S \cdot \Delta_{U,s} = 0$ for all $U \in \Inj(\lambda, \leq n)$. We have that
    $$ 
    \partial F_T^S \cdot \Delta_{U,s} = \sum_{D \in \DX_{\col}(U,s)}\sum_{P\in \snake(T)} (-1)^D(-1)^P \partial \xx_P^S \cdot \xx_D
    $$
    
    We will do this by constructing a sign-reversing, weight-preserving involution on the set $\DX_{\col}(U,s) \times \snake(T)$. It suffices to consider the case where $\partial \xx_P^S \cdot \xx_D$ is nonzero. We write $\ax(P,S),\bx(D)$ for the exponent tuples (as in \eqref{eq: exponent-tuple-def}).
    
    First, for sake of contradiction, we assume that there does not exist a column of $D$ of the form $\ytableaushort{\beta,\alpha}$ such that $\alpha,\beta$ appear in the same row of $P$ and $\ax_\alpha = \bx_\alpha, \ax_\beta = \bx_\beta$: that is, $\ax_\alpha = 0, \ax_\beta =1$. We will show that such pair exists, and use it to construct the sign reversing involution.

    Note that since $D$ consists of $(\lambda_1,\lambda_2)$ along with the horizontal strip $\xi$, we have that 
    \[\bx_i = \begin{cases}
        0 &  i \text{ appears in row 1 of }D,\\ 
        1 &  i \text{ appears in row 2 of }D,\\ 
        s-1 &  i \text{ appears in the horizontal strip } \xi.
    \end{cases}\]

    Since $\partial \xx_P^S \cdot \xx_D\neq 0$ we know that $\mathbf{a}_i \leq \mathbf{b}_i$ for all $i$ by Lemma \ref{lem: derivative-def}. 
    Note that $\ax_i$ is the entry in the box of $\cc(S)$ corresponding to box in $P$ containing $i$.
    Since 0 can only appear in the first row of $\cc(S)$, we have that if $\bx_i =0$, then $i$ appears in the first row of $P$ and the corresponding entry in $\cc(S)$ is 0. Thus we have at least $\lambda_1$ many $0$'s in row 1 of $\cc(S)$. From Algorithm \ref{alg: ctype}, this implies
    $\lambda_1 \leq \ctype(S\vert_k)_1$.

    Similarly, if $\bx_i= 1$, we know that $i$ must appear in row $1$ or $2$ of $P$, since the only rows where $1$ can appear in $\cc(S)$ are rows 1 and 2.
    Assume there exists a $\beta$ such that $\bx_\beta =1$ and $\beta$ appears in row $1$ of $P$.  Then we would have a column $\ytableaushort{\beta,\alpha}$ of $D$ such that $\alpha,\beta$ appear in the same row and $\ax_\alpha = 0$. Thus, by assumption, we must have $\ax_\beta < 1$.

    Thus, if $\bx_i = 1$, we must have one of two cases:
    \[\begin{cases}
        i \text{ appears in row 1 of } P \text{ and the corresponding entry in }\cc(S) \text{ is } 0,\\
        i \text{ appears in row 2 of } P \text{ and the corresponding entry in }\cc(S) \text{ is } 1.
    \end{cases}\]

   By Algorithm \ref{alg: ctype}, we have that 
   $\lambda_1+\lambda_2 \leq \ctype(S\vert_{k})_1 +  \ctype(S\vert_{k})_2$, which contradicts 
    $(\lambda_1,\lambda_2) \ntrianglelefteq	\ctype(S|_k)$.

    Thus, there exists a column of $D$ of the form $\ytableaushort{\beta,\alpha}$ such that $\alpha,\beta$ appear in the same row of $P$ and $\ax_\alpha = 0, \ax_\beta =1$. 
    Choose $(\alpha,\beta)$ to the from the leftmost such column across all such pairs; see Example \ref{ex: sri-pair}.
    
    Let $\tau = (\alpha \,\, \beta)$ denote the \emph{column} transposition on $D$ and $\sigma = (\alpha \,\, \beta)$ denote the \emph{row} transposition on $P$ that swaps $\alpha$ and $\beta$. Define $\Phi: \DX_{\col}(U,s) \times \snake(T) \to \DX_{\col}(U,s) \times \snake(T)$ by $\Phi( D,P) \to \Phi(\tau D,\sigma P)$. Since $\tau$ is a transposition, we have that
    $$ \sgn(\Phi(D,P)) =  (-1)^{\tau D} (-1)^{\sigma P} = (-1)(-1)^D(-1)^P = -\sgn((D,P))$$
    This has the effect of swapping $\ax_\alpha,\ax_\beta$ in $\ax$, and $\bx_\alpha,\bx_\beta$ in $\bx$. But since $\ax_\alpha = \bx_\alpha,\ax_\beta = \bx_\beta$ the quantities
    $$ x_\alpha^{\bx_\alpha - \ax_\alpha} \qquad x_\beta^{\bx_\beta - \ax_\beta} \qquad \frac{\bx_\alpha! \bx_\beta!}{(\bx_\alpha-\ax_\alpha)! (\bx_\beta-\ax_\beta)!}$$
    all remain unchanged, so $\Phi$ is weight preserving. It is easy to see that $\Phi$ is an involution. So we have $\partial F_T^S \cdot \Delta_{U,s} = 0$, so that $F_T^S = 0$ in $R_{n,\lambda,s}$.
\end{proof}

\begin{example}\label{ex: sri-pair}

    Let $\lambda = (4,3), n = 9, s = 2, S = \ytableaushort{56,1234789}$. Note that $\ctype(S|_7) = (4,2,1) \not\trianglerighteq (4,3)$. We give an example of the sign-reversing involution $\Phi$. Let
    $$
    \begin{ytableau}
        2 & 5 & 7 & \none & \none & 8 & 9\\
        1 & 4 & 6 & 4
    \end{ytableau}
    \qquad
    T = \ytableaushort{25,1346789}
    $$
    One such nonzero pair $(D,P) \in \DX_{\col}(U,s) \times \snake(T)$ is given by
    $$
    D =
    \begin{ytableau}
        1 & 3 & *(pink)7 & \none & \none & 8 & 9\\
        2 & 5 & *(pink)6 & 4
    \end{ytableau}
    \qquad
    P = 
    \begin{ytableau}
        1 & 3 \\
        5 & *(pink)6 & 2 & 4 & *(pink)7 & 9 & 8
    \end{ytableau}
    \qquad 
    \cc(S) =
    \begin{ytableau}
        1 & 1 \\
        0 & *(pink)0 & 0 & 0 & *(pink)1 & 1 & 1
    \end{ytableau}
    $$
    so that $\ax(P,S) = (1,0,1,0,0,0,1,1,1)$ and $\bx(D) = (1,0,1,0,0,0,1,1,1)$. We have $i = 3, \alpha = 6, \beta = 7$, which is highlighted in pink in the previous diagram, and so $\Phi$ simply swaps this pair:
    $$
    \Phi( D,P) =
    \bigg(
    \begin{ytableau}
        1 & 3 & *(pink)6 & \none & \none & 8 & 9\\
        2 & 5 & *(pink)7 & 4
    \end{ytableau},
    \qquad
    \begin{ytableau}
        1 & 3 \\
        5 & *(pink)7 & 2 & 4 & *(pink)6 & 9 & 8
    \end{ytableau}
    \bigg)
    $$
    Note this changes the sign contribution of $D$, but not of $P$, and the monomial is preserved.
\end{example}

\begin{rem}
     Proposition \ref{prop: lambda-tilde-bad-death} fails for general $\lambda$, with the first counterexample appearing for $n=k=6$, $\lambda = (2,2,2)$. While \emph{most} higher Specht polynomials $F_T^S \not\in \CC\nls$ do indeed vanish modulo $I\nls$, there are some that are nonzero in $R\nls$, yet contained in the span of $\CC\nls$.
\end{rem}




Now, we can show that our set gives a basis for $(n,(\lambda_1,\lambda_2),\ell(\lambda))$.
\begin{proof}[Proof of Theorem \ref{thm: higherspecht}]
    Note that from Lemma \ref{lem:containment oif ideals} we have that $I_{n,\nu,\ell(\lambda)} \subset I_{n,\lambda,\ell(\lambda)}$, which implies there exists a surjection $R_{n,\nu,\ell(\lambda)}  \twoheadrightarrow R_{n,\lambda,\ell(\lambda)}$ for $\nu$ that appears in Corollary \ref{cor: lambda-tilde-span}. Combining this with Corollary \ref{cor: lambda-tilde-span}, we have that $\widetilde{\CC}_{n,\ell(\lambda)}$ spans $R_{n,\lambda,\ell(\lambda)}$. From Proposition  \ref{prop: lambda-tilde-bad-death}, the elements in $\mathcal{C}_{n,\lambda,\ell(\lambda)}$ that are not in $\widetilde{\CC}_{n,\ell(\lambda)}$ are exactly those that are 0 in $R_{n,\lambda,\ell(\lambda)}$. Thus $\mathcal{C}_{n,\lambda,\ell(\lambda)}$ still spans $R_{n,\lambda,\ell(\lambda)}$. The statement follows from Corollary \ref{cor: count}, which shows that $|\CC\nls| = \dim(R\nls)$.
\end{proof}

Specializing to the case $(n,\lambda,s) = (k,(\mu_1,\mu_2),\ell(\mu))$ gives a higher Specht basis for the Garsia--Procesi module $R_\mu$.
\begin{cor}
For $\mu = (\mu_1,\mu_2)$, the polynomials 
\begin{equation*}\label{eq: higher specht of gp}
    \CC_\mu := \{ F_T^S: S \in \Sigma_\mu,  \sh(S) = \sh(T)\}
\end{equation*}
 form a higher Specht basis for the Garsia-Procesi module $R_\mu$.
\end{cor}

\begin{rem}
    To prove our conjecture holds for all $(n,\lambda,s)$, it suffices to show $\CC_{n,\lambda,\ell(\lambda)}$ is a basis for $R_{n,\lambda,\ell(\lambda)}$. We can then extend the result to all $s\geq \ell(\lambda)$ by induction using the exact sequence in Griffin \cite{griffin2021ordered}. We can extend Proposition \ref{prop: lambda-tilde-bad-death} to show that \emph{most} non-basis elements vanish; it would suffice to express the non-vanishing, non-basis elements as linear combinations of $\CC_{n,\lambda,\ell(\lambda)}$.
\end{rem}

\subsection{Connection to Gillespie-Rhoades Bases}\label{subsec: gillespie-rhoades}
Gillespie-Rhoades \cite{RhoadesGillespie2021} conjectured that the following set gives a higher Specht basis for the Garsia-Procesi module: 

\begin{conjecture}[Conjecture 3.6, Gillespie--Rhoades\cite{RhoadesGillespie2021}]\label{conj: gillespie-rhoades}
    The set
    $$ \CC_\mu' := \bigcup\limits_{\lambda\vdash n} \{ F_T^S : (S,T) \in \SSYT(\lambda,\mu) \times \SYT(\lambda) \}$$
    descends to a higher Specht basis of $R_\mu$.
\end{conjecture}
For the precise definition of $F_T^S$ in the case where $S$ is a semistandard tableaux, see \cite{RhoadesGillespie2021}.
We can see that the definition of $\mathcal{C}'_\mu$ differs from our higher Specht basis given in \eqref{eq: higher specht of gp}. 
In particular, the set $\mathcal{C}'_\mu$ uses cocharge tableaux arising from semistandard tableaux of weight $\mu$, rather than the corresponding subset of standard tableaux $\Sigma_\mu$ that it gets mapped to with respect to $\std$. 
Though the two formulas seem similar, they yield different sets, since in general, the standardization map $\std$ only preserves the cocharge statistic, not the cocharge tableaux,. 

For example, consider 
 \[\ytableausetup{smalltableaux} S =  \ytableaushort{4,24, 11233} \ \ \  \std(S) =\ytableaushort{7,68,12345}\]  

The corresponding cocharge tableaux are
 \[\ytableausetup{smalltableaux} \cc(S) =  \ytableaushort{1,12, 00001} \ \ \  \cc(\std(S)) =\ytableaushort{2,12,00000}\]  

A natural question to ask if whether this other set $\mathcal{C}'_\mu$ also forms a higher Specht basis of $R_\mu$. We answer this question using an example.  Consider the following tableau $S \in \SSYT((4,4,1,1),(2,2,2,2,2))$:

$$ S = \ytableaushort{5,3,2345,1124} \qquad \cc(S) = \ytableaushort{2,2,1113,0002}$$

Let $T \in \SYT((4,4,1,1))$, and $P \in \snake(T)$. We must have that the first snake $P_1$ passes through both $2$'s. Denote the boxes in the second and third rows (each consisting only of one box) by $u,v$ respectively. We must then necessarily have that $F_T^S = 0$ for all $T$, since we may construct a sign-reversing involution on $\snake(T)$ that simply swaps the entries of $u,v$ at the cost of a sign.
From this, we can see that $\mathcal{C}'_\mu$ does not form a basis of $R_\mu$ in general. 

However, the authors show that $\mathcal{C}'_\mu$ is indeed a higher Specht basis when $\mu = (\mu_1, \mu_2)$. This is no coincidence: in this case, we have that $\mathcal{C}'_\mu$ is exactly our higher Specht basis $\mathcal{C}_\mu$. 

To see this, we describe the standardization map $\std: \SSYT(\mu)\to \Sigma_\mu$.
In general, the standardization map can be described to be the composition of the crystal reflection operators $s_i$ and the lowering operator $f_1$ (see \cite{SW}). For $(\mu_1,\mu_2)$, we can easily describe this map explicitly. 
The map $\std$ takes the semistandard tableau of shape $\lambda = (\lambda_1,\lambda_2)$ of weight $\mu$ and maps it to the standard tableau obtained in the following way:
\begin{itemize}
    \item Replace the 1s in the first row with $12\dots \mu_1$,
    \item replace the 2s in the second  row with $(\mu_1)\dots (\mu_1+ \lambda_2)$,
      \item replace the 2s in the second row with $(\mu_1+\lambda_2) \dots n$.
\end{itemize}

The following proposition is an immediate consequence of the standardization map. 
\begin{prop}
    Let $\mu = (\mu_1,\mu_2)$. Then for any $S\in \SSYT(\mu)$, we have $\cc(S) = \cc(\std(S))$.
\end{prop}

\begin{proof}
Let $S\in \SSYT(\lambda,\mu)$, where $\lambda,\mu$ are both two row partitions.
We see that $\cc(S)$ consists of all 0s in the first row and  1s in the second row.
From the description of $\std(S)$, we can see that the cocharge tableau of $\cc(\std(S))$ is also the same.
\end{proof}

From this, we get that the two constructions coincide for two row partitions.
\begin{cor}
    For $\mu = (\mu_1, \mu_2)$, we have that $\mathcal{C}_\mu = \mathcal{C}'_\mu.$
\end{cor}

\bibliographystyle{plain}
\bibliography{refs}

\end{document}